\documentclass[11pt]{amsart}
\usepackage{amsmath,amsthm,amssymb,amsfonts}
\usepackage[margin=1in]{geometry}
\usepackage{mathrsfs}
\usepackage{tikz-cd}
\usepackage{pdflscape}
\usepackage{makecell}
\usepackage{longtable}
\usepackage{xcolor}
\usepackage{multirow}
\usepackage{hyperref}
\usepackage{cleveref}

\hypersetup{
    colorlinks=true,
    linkcolor=red,     
    citecolor=blue,     
    urlcolor=blue       
}

\newtheorem{theorem}{Theorem}[section]
\newtheorem{lemma}[theorem]{Lemma}
\newtheorem{proposition}[theorem]{Proposition}
\newtheorem{corollary}[theorem]{Corollary}
\newtheorem{conjecture}[theorem]{Conjecture}

\newtheorem{thmA}{Theorem}

\newtheorem{corA}[thmA]{Corollary}

\theoremstyle{definition}
\newtheorem{definition}[theorem]{Definition}

\theoremstyle{definition}
\newtheorem{remark}[theorem]{Remark}

\newcommand{\Z}{\mathbb{Z}}

\newcommand{\Q}{\mathbb{Q}}
\newcommand{\R}{\mathbb{R}}

\newcommand{\PP}{\mathbb{P}}
\newcommand{\OO}{\mathcal{O}}

\newcommand{\X}{\underline{X}}
\newcommand{\cX}{\mathcal{X}}
\newcommand{\Nef}{\operatorname{Nef}}

\title[Del Pezzo Orbifolds]{Campana rational connectedness and Weak approximation of del Pezzo orbifolds}
\author{Saptarshi Dandapat}
\address{Graduate School of Mathematics, Nagoya University, Furocho Chikusa-ku, Nagoya, 464-8601, Japan}
\email{saptarshi.dandapat.d2@math.nagoya-u.ac.jp}
\date{\today}

\begin{document}

\begin{abstract}
    We prove that a del Pezzo orbifold of degree less than 7, and del Pezzo orbifolds of higher degree with irreducible boundary are strongly Campana 
    uniruled, and deduce that it is Campana rationally connected. This provides important non-toric examples in dimension two of a conjecture by Campana, 
    in the form required by the weak approximation theorem, and yields weak approximation for Campana sections of Campana fibrations at all places 
    whose general fibre is such an orbifold. In three appendices we classify the del Pezzo orbifolds of degree $6$, $7$ and $8$ 
    with irreducible boundary.
\end{abstract}

\maketitle 

\tableofcontents

\section{Introduction}

\subsection{Campana points and sections of fibrations}

Let $k$ be an algebraically closed field of characteristic $0$ and let $B$ be a smooth projective curve over $k$ with function field $K = k(B)$. 
By the valuative criterion of properness, $K$-rational points of a variety defined over $K$ are identified with the sections of a proper model over $B$. 
When the generic fibre is rationally connected such a section always exists (\cite[Theorem~1.1]{graber2003families}), and they satisfy 
weak approximation at the places of good reduction \cite[Theorem~3]{hassett2005}.  

In arithmetic algebraic geometry a recent point of interest has been \emph{Campana points}, which interpolate between rational points and integral points
and are described by the geometry of a Campana orbifold $(X, \Delta_{\epsilon})$ in the sense of Campana \cite{campana2004orbifolds,Campana2011,Campana2011Survey}.
Many conjectures for rational points and integral points have been extended and studied for Campana points (see \cite{abramovich2018campana,pieropan2021campana,mitankin2024semi}).
Over the function field $K$ the object corresponding to a Campana point is a \emph{Campana section} of a Campana fibration 
$\pi : (\cX,\Delta_\epsilon) \to B$, that is, a section whose local intersection multiplicities with the boundary satisfy the Campana condition.

Chen, Lehmann and Tanimoto \cite{chen2024campana} developed the deformation theory of Campana curves and Campana sections within logarithmic geometry,
using moduli of stable log maps, and proved a version of \cite{hassett2005} in this setting: if the general fibre of $\pi$ is rationally connected 
\emph{and strongly Campana uniruled}, then weak approximation for Campana sections holds at the places of good reduction 
\cite[Theorem~1.6]{chen2024campana}. The theorem becomes useful only when its hypothesis on the general fibre can be checked in concrete families. 
Here we check it for del Pezzo orbifolds.

\subsection{Campana's conjecture on Campana rational connectedness.}\label{ss:conjectures}

Throughout, a \emph{klt Campana orbifold} is a pair $(X, \Delta_{\epsilon})$ with underlying smooth projective variety $\X$, 
a strict normal crossings divisor $\Delta = \bigcup_i \Delta_i$ on $\X$, and weights $\epsilon_i = 1 - 1/m_i$ with $m_i \in \Z_{\geq 1}$, 
giving the $\Q$-divisor $\Delta_{\epsilon} = \sum_i \epsilon_i \Delta_i$. It is called a \emph{klt Fano orbifold} if $-(K_{\X} + \Delta_{\epsilon})$ is ample. 
A rational Campana curve is a genus zero stable log map to the log scheme attached to $(X,\Delta_{\epsilon})$ whose contact order along $\Delta_i$ 
at each marking is either $0$ or at least $m_i$, for precise definitions see Section \ref{sec:prelim}. 

In \cite[Section~5.4]{Campana2011} Campana formulated several conjectures relating Campana uniruledness to the positivity of the orbifold tangent bundle. 
We are interested in a special case of \cite[Conjecture~9.10]{Campana2011Survey}, which is the main source of examples of Campana rational connectedness.

\begin{conjecture}[Campana; {\cite[Conjecture~5.8]{chen2024campana}}]
    \label{conj:campana}
    Let $k$ be of characteristic $0$ and let $(X,\Delta_{\epsilon})$ be a klt Fano orbifold. Then $(X,\Delta_{\epsilon})$ is Campana rationally connected.
\end{conjecture}

\begin{conjecture}[{\cite[Conjecture~5.10]{chen2024campana}}]
    \label{conj:CLT}
    Let $k$ be of characteristic $0$. Then every klt Fano orbifold $(X,\Delta_{\epsilon})$ is strongly Campana uniruled, i.e.\ carries a free Campana
    curve $f : C \to X$ whose class $f_*[C]$ lies in the interior of the nef cone of curves $\Nef_1(X)$.
\end{conjecture}

The cone $\Nef_1(X)$ is dual to the pseudo-effective cone of divisors, so $f_*[C]$ lie in its interior implies:
$$f_*[C] \cdot D > 0 \qquad \text{for every non-zero effective divisor } D.$$ 
In particular the curve meets every boundary component $\Delta_i$, and after splitting contact orders we obtain, for each $i$, a free Campana curve
through a general point of the codimension one stratum of $\Delta_i$ \cite[Propositions~2.11 and 5.11]{chen2024campana}. When $\X$ is rationally 
connected, Conjecture \ref{conj:CLT} implies Conjecture \ref{conj:campana} \cite[Corollary~6.7]{chen2024campana}.

Conjecture~\ref{conj:CLT} is known in the following cases: when $\operatorname{dim} X = 1$, by an analysis of the finite subgroups of
$\mathrm{PGL}_2$ and the associated Bely\u{\i} maps \cite[Theorem~7.1]{chen2024campana}; and when $X$ is a smooth projective toric variety 
and $\Delta$ is the torus-invariant boundary \cite[Theorem~1.8]{chen2024campana}. Campana's cyclic cover construction \cite[Example~1.5]{chen2024campana} 
produces many klt Fano orbifolds that are Campana rationally connected. 
In this paper, we prove the Conjecture~\ref{conj:CLT} and Conjecture~\ref{conj:campana} for del Pezzo orbifolds. 

\subsection{Weak approximation for Campana sections of del Pezzo fibrations}\label{ss:introweak}

Let $k$ be an algebraically closed field of characteristic $0$, $B$ is a smooth projective curve over $k$ and $K = k(B)$. For a closed point $b \in B$
we write $\widehat{\OO_b}$ for the completion of the local ring $\OO_{B,b}$ and $K_b = \operatorname{Frac}(\widehat{\OO_b})$ for the completion of $K$
at the corresponding discrete valuation $v_b$, we refer to closed points of $B$ as \emph{places} of $K$. For a smooth geometrically integral $K$-variety
$X$ the set $X(K_b)$ carries a natural $v_b$-adic topology. 

\begin{definition}
    \label{def:weak-approximation}
    Let $X$ be a smooth geometrically integral $K$-variety and let $V$ be a finite set of places of $K$. We say that $X$ satisfies 
    \emph{weak approximation at $V$} if the image of the diagonal map $$X(K) \to \prod\limits_{b \in V} X(K_b)$$ is dense in the product topology.
    We say that $X$ satisfies \emph{weak approximation} if it satisfies weak approximation at $V$ for every finite set $V$ of places of $K$.
\end{definition}

Let $\pi: \cX \to B$ be a proper model of $X$, that is, a proper flat morphism from an integral scheme with generic fiber $X$. By valuative criterion
of properness, $X(K)$ is the set of sections of $\pi$ and $X(K_b)$ is the set of sections of 
$\cX \times_B \operatorname{Spec}(\widehat{\OO_b}) \to \operatorname{Spec}(\widehat{\OO_b})$. Truncating a section at order $n$ produces an admissible $n$-jet 
of $\pi$ at a point of the fibre $\cX_b$, and conversely every admissible jet arises in this way. (See \cite[Section~6]{chen2024campana} for an Introduction 
to weak approximation for Campana sections.) For del Pezzo surfaces of degree at least $4$ over the function field of a curve, weak approximation is known
at \emph{every} place by a theorem of Colliot-Th\'el\`ene and Gille \cite{colliot2004remarques}. Combining this with our results we obtain our main 
arithmetic application towards weak approximation for Campana sections at \emph{every} place.  

\subsection{Main results}

\begin{definition}
    \label{def:dP-orbifold}
    A \emph{del Pezzo orbifold} is a klt Campana orbifold $(X,\Delta_{\epsilon})$ such that $\X$ is a smooth del Pezzo surface and 
    $-(K_{\X} + \Delta_{\epsilon})$ is ample. Its \emph{degree} is $d = K_{\X}^2$.
\end{definition}

\begin{thmA}
    \label{thmA}
    Let $k$ be algebraically closed of characteristic $0$ and let $(X, \Delta_{\epsilon})$ be a del Pezzo orbifold. 
    If $d \leq 6$ or $\X = \PP^1 \times \PP^1$, such that $-(K_{\X} + \Delta_{\epsilon})$ is ample then $(X, \Delta_{\epsilon})$ is strongly Campana uniruled.
    When $\X$ is $\PP^2$, blow-up of $\PP^2$ at one, or two points, and $\Delta$ is irreducible, such that $-(K_{\X} + \Delta_{\epsilon})$ is ample 
    then $(X, \Delta_{\epsilon})$ is strongly Campana 
    uniruled, that is, Conjecture \ref{conj:CLT} holds for $(X,\Delta_{\epsilon})$.
\end{thmA}

\begin{thmA}
    \label{thmB}
    Let $k$ be algebraically closed of characteristic $0$ and let $(X, \Delta_{\epsilon})$ be a del Pezzo orbifold. 
    If $d \leq 6$ or $\X = \PP^1 \times \PP^1$, such that $-(K_{\X} + \Delta_{\epsilon})$ is ample then $(X, \Delta_{\epsilon})$ is Campana rationally connected.
    When $\X$ is $\PP^2$, blow-up of $\PP^2$ at one, or two points, and $\Delta$ is irreducible, such that $-(K_{\X} + \Delta_{\epsilon})$ is ample 
    then $(X, \Delta_{\epsilon})$ is Campana rationally 
    connected, that is, Conjecture \ref{conj:campana} holds for $(X,\Delta_{\epsilon})$.
\end{thmA}

Theorem \ref{thmA} is proved as Theorem \ref{thm:free-campana-curve} and Theorem \ref{thmB} follows from it together with 
\cite[Corollary~6.7]{chen2024campana}, since a del Pezzo surface is rationally connected. Using
\ref{thmA} and \cite[Theorem~6.4]{chen2024campana} we get the arithmetic consequence that motivated the whole discussion.

\begin{corA}
    \label{corC}
    Let $k$ be algebraically closed of characteristic $0$, let $B$ be a smooth projective curve over $k$, and let $\pi : (\cX, \Delta_{\epsilon}) \to B$ 
    be a klt Campana fibration whose general fibre with its induced orbifold structure is a low degree $(\leq 6)$ del Pezzo orbifold, or a higher 
    degree del Pezzo orbifold with irreducible boundary. 
    Let $S \subset B$ be a finite set of closed points such that the underlying fibration $\underline{\pi} : \underline{\cX} \to B$ satisfies 
    weak approximation away from $S$, by \cite[Theorem~3]{hassett2005} one may take $S$ to be the set of places of bad reduction. 
    Then every finite collection of Campana jets supported on distinct fibres over $B \setminus S$ is induced by a Campana section of $\pi$.
\end{corA}

\begin{proof}
    A general fibre of $\pi$ is a del Pezzo surface, hence rationally connected, and is strongly Campana uniruled by Theorem \ref{thmA}. Now apply
    \cite[Theorem~6.4]{chen2024campana}.
\end{proof}

Finally, we record the classification underlying the appendices. For $d \leq 5$ this is \cite[Theorem~1.2]{dandapat2026classification}. The cases 
$d = 6, 7, 8$ are proved here in Appendices \ref{app:six}-\ref{app:eight}, and the case $d = 9$ is elementary (Proposition \ref{prop:P2}).

\begin{thmA}
    \label{thmD}
    Let $\X$ be a del Pezzo surface of degree $d$ and let $D$ be an irreducible curve on $\X$, with $\epsilon = 1 - 1/m$ and $m \geq 2$. 
    The pairs $(\X, \epsilon D)$ for which $-(K_{\X} + \epsilon D)$ is ample, and those for which it is only nef, are listed explicitly in 
    \textup{\cite[Theorem~1.2]{dandapat2026classification}} for $d \leq 5$, in Propositions \ref{prop:dp6},  \ref{prop:dp7}, \ref{prop:dp8-1} and 
    \ref{prop:dp8-2} for $d = 6,7,8$, and in Proposition \ref{prop:P2} for $d = 9$. In particular $-K_{\X} \cdot D \leq 2d$ in every such case.
\end{thmA}

Note, as in Remark \ref{rmk:classification-not-used} below, that Theorems \ref{thmA} and \ref{thmB} use neither \cite{dandapat2026classification} 
nor the appendices: the ampleness inequality $\Delta_{\epsilon} \cdot F < -K_{\X} \cdot F$ replaces the case by case verification. We include the 
classification because they are of independent interest and it makes the class of orbifolds under discussion explicit.

\begin{thmA}
    \label{thmE}
    Let $k$ be an algebraically closed field of characteristic $0$, let $B$ be a smooth projective curve over $k$, and let 
    $\pi: (\cX, \Delta_{\epsilon}) \to B$ be a klt Campana fibration whose general fibre, with its induced orbifold structure, is a del Pezzo orbifold
    of degree at least $3$ (with irreducible boundary for degree $\geq 7$). Then every finite collection of admissible Campana jets supported on distinct fibres of $\pi$ 
    is induced by a Campana section of $\pi$. Equivalently, weak approximation for Campana sections of $\pi$ holds at every place of $B$, including the
    places of bad reduction.
\end{thmA}

\subsection{Strategy of the proof}\label{ss:strategy}

Conjecture~\ref{conj:CLT} is a theorem in dimension one \cite[Theorem~7.1]{chen2024campana}. Let $F \in \operatorname{Pic}(\X)$ be a conic class, that is, $F$ effective with 
$F^2 = 0$ and $-K_{\X} \cdot F = 2$, by Lemma \ref{lem:conic-class} the linear system $|F|$ is a base point free pencil whose general member 
$\underline{C}'$ is a smooth rational curve. Ampleness of $-(K_{\X}+\Delta_{\epsilon})$ says that
\begin{equation}
\label{eq:key-inequality}
  \sum_{k} \Bigl( 1 - \frac{1}{m_{i(k)}} \Bigr)
  \;=\; \Delta_{\epsilon} \cdot F \;<\; -K_X \cdot F \;=\; 2 ,
\end{equation}
the sum being over the points of $\underline{C}' ~\cap~ \Delta$, weighted by the multiplicity of the boundary component through each of them. 
The inequality \eqref{eq:key-inequality} says that the orbifold $\PP^1$ is spherical, and a cover $g : \PP^1 \to \underline{C}'$ 
with the required ramification exists (Lemma \ref{lem:covers}, a variant of \cite[Theorem~7.1]{chen2024campana}). The composition $f = h \circ g$ is 
then a Campana curve, where $h: C' \hookrightarrow X$ is an inclusion, and a computation of the degree of its log normal bundle (Lemma \ref{lem:degree}) shows that it is free. We emphasize on the 
following two points:

\begin{enumerate}
    \item[(i)] The cover $g$ is \emph{not} taken to be totally ramified. A cover of $\PP^1$ totally ramified over $n$ marked points exists only for 
    $n \leq 2$, whereas $\Delta \cdot F$ can be larger, \eqref{eq:key-inequality} gives a required cover with ramification indices $e_q \geq m_{i(k)}$ 
    over the $k$-th point. For $m = 2$ this permits up to three boundary points, and the relevant covers are quotients by dihedral, tetrahedral, 
    octahedral or icosahedral subgroups of $\mathrm{PGL}_2(k)$.

    \item[(ii)] The number $n$ of boundary points on $\underline{C'}$ is the intersection number $\Delta \cdot F$ because by Bertini's theorem we
    choose $\underline{C'}$ meeting $\Delta$ transversally and away from $\operatorname{Sing}(\Delta)$. When a tangency is introduced as in Case 2 of 
    Proposition \ref{prop:line}, where $\underline{C'}$ is the strict transform of a general tangent line of $\beta_* D$, one has $n = t-1$ with
    $t = D \cdot H$, and the multiplicities $\mu_k$ of Lemma \ref{lem:degree} must be carried through the computation. Tangency relaxes the 
    hypothesis of Lemma \ref{lem:covers} and strengthen that of Lemma \ref{lem:degree}. The constructions in Section \ref{sec:mainthm} satisfies the both.
\end{enumerate}

Conic classes do not suffice for the interiority statement. If $\rho(X) \leq 2$ and $X \neq \PP^1 \times \PP^1$, the conic classes span a subcone, 
and we must also realise the pullback $H$ of the line class by a free Campana curve. This is where the tangent line construction of 
(ii), and the hypothesis that the boundary be irreducible, are used. Lemma \ref{lem:positivity} then isolates a finite set $\mathcal{G}(X)$ of nef classes such that
any positive integral combination of them lies in the interior of $\Nef_1(X)$ and any two of them meet. Choosing representatives that meet away from
$\Delta$, gluing and smoothing (Proposition \ref{prop:gluing}) produces a single free Campana curve of class $\sum_j N_j \gamma_j$, which proves 
Theorem \ref{thmA}. Theorem \ref{thmB} and Corollary \ref{corC} then follow from \cite{chen2024campana} as indicated above.

\begin{remark}
    \label{rmk:sharp}
    The only place where ampleness of $-(K_{\X}+\Delta_{\epsilon})$ is used is the strict inequality \eqref{eq:key-inequality}. If 
    $-(K_{\X}+\Delta_{\epsilon})$ is merely nef then $\Delta_{\epsilon} \cdot F = 2$ may occur, the induced orbifold $\PP^1$ is Euclidean rather than 
    spherical, and no cover $\PP^1 \to \PP^1$ with the required ramification exists, only covers by elliptic curves do. The construction therefore 
    detects exactly the log Fano condition. 
\end{remark}

\begin{remark}
    \label{rmk:scope}
    If $(X,\Delta_{\epsilon})$ is a klt Fano orbifold with $\X$ a smooth projective surface, then $-K_{\X}= -(K_{\X}+\Delta_{\epsilon}) + \Delta_{\epsilon}$ 
    is big but need not be ample, so $\X$ need not be a del Pezzo surface. For instance, let $\X = \mathbb{F}_2$ be the Hirzebruch surface with negative section 
    $C_0$, $C_0^2 = -2$, and let $\Delta_{\epsilon} = \tfrac12 C_0$, i.e.\ $m = 2$. Then the intersection number of 
    $-(K_{\X} + \tfrac12 C_0) = \tfrac32 C_0 + 4f$ with $C_0$ is  $1$ and with the ruling $f$ is $\tfrac32$, hence it is ample, while $-K_{\X} \cdot C_0 = 0$. 
    Thus Theorems \ref{thmA} and \ref{thmB} do not complete Conjecture \ref{conj:CLT} in dimension two, the remaining surfaces are the non-del Pezzo 
    smooth rational surfaces, where the boundary contains the curves of non-positive anticanonical degree.
\end{remark}
     
\begin{remark}[Reducible boundaries]
    \label{rmk:reducible}
    Lemma \ref{lem:degree}, Proposition \ref{prop:conic} are proved for an arbitrary strict normal crossings boundary 
    $\Delta_{\epsilon} = \sum_i \epsilon_i \Delta_i$. Consequently Theorem \ref{thmA} holds verbatim for every klt Campana orbifold whose underlying 
    surface is a del Pezzo surface of degree at most $6$, or is $\PP^1 \times \PP^1$. Only the auxiliary class $H$ of Proposition \ref{prop:line} uses 
    irreducibility, so only $\PP^2$, the blow up of $\PP^2$ at one point, and the del Pezzo surface of degree $7$ require a separate argument in the 
    reducible case, and that argument cannot be avoided, as Remark \ref{rem:reducible} shows.
\end{remark}
     
\subsection{Organisation of the paper}
     
    In section~\ref{sec:prelim} we fix notation and recall results from \cite{chen2024campana}: Campana curves and their contact orders 
    (\S\ref{ss:campana-curves}), the normal complex and free log curves (\S\ref{ss:freeness}), gluing and smoothing (\S\ref{ss:gluing}), weak 
    approximation (\S\ref{ss:WA}), and the geometry of del Pezzo surfaces that we use (\S\ref{ss:dP}). Section~\ref{sec:mainthm} contains the proofs: 
    \S\ref{ss:lemmas} the two elementary lemmas which control the construction of free Campana curve, \S\ref{ss:construction} the production of free 
    Campana curves from conic classes and from the line class, and \S\ref{ss:proofs} the main theorems. Section~\ref{sec:wa} contains the arithmetic 
    results of weak approximation for Campana sections: \S\ref{ss:dp-fib} defines Campana del Pezzo fibration and \S\ref{ss:thmE} contains the proof
    of Theorem~\ref{thmE}. Appendices \ref{app:six}, \ref{app:seven} and \ref{app:eight} classify the del Pezzo orbifolds with irreducible boundary 
    of degree $6$, $7$ and $8$ respectively.
     
\subsection*{Conventions}
     
    We work over an algebraically closed field $k$. From \S\ref{ss:freeness} onwards we assume that $\operatorname{char} k = 0$, which is where 
    the results of \cite{chen2024campana} are proven and we use Bertini, generic smoothness and biduality. All log structures are fine and saturated.
    A \emph{del Pezzo surface} is a smooth projective surface with $-K_X$ ample, and its degree is $K_X^2$. A \emph{$(-1)$-curve} is an irreducible 
    curve $E$ with $E^2 = -1$ on a del Pezzo surface.

\subsection*{Acknowledgement} The author is grateful to his advisor Prof. Sho Tanimoto for introducing the Conjecture~\ref{conj:campana} and for his 
encouragement, useful discussions on the first draft of the paper. The author thanks Sho Tanimoto, Brian Lehmann and Qile Chen for their comments 
on the first draft of the paper. The author was supported by Japanese Government MEXT Scholarship for Research Students, recommended by 
Embassy of Japan, New Delhi, India. Claude Opus 5 was used to proofread the document. 

\section{Preliminaries}{\label{sec:prelim}}

In this section we recall results from \cite{chen2024campana} about logarithmic deformation theory. For the foundations of logarithmic geometry we refer
to \cite{kato1988logarithmic} and to the book \cite{ogus2018lectures}, and for the theory of stable log maps to \cite{chen2014stable, abramovich2014stable,gross2013logarithmic}.

\subsection{Notation}
\label{ss:notation}
 
For a smooth projective variety $\X$ we write $N^1(\X)$ and $N_1(\X)$ for the groups of divisors and of $1$-cycles modulo numerical equivalence, and
$\overline{\operatorname{Eff}}(\X) \subset N^1(\X)_{\R}$ and $\Nef_1(\X) \subset N_1(\X)_{\R}$ for the pseudo-effective cone of divisors and its dual cone, the \emph{nef cone of
curves}. Thus a class $\alpha \in N_1(\X)_{\R}$ lies in the interior of $\Nef_1(\X)$ if and only if $\alpha \cdot D > 0$ for every non-zero $D \in \overline{\operatorname{Eff}}(\X)$, 
and, when $\overline{\operatorname{Eff}}(\X)$ is rational polyhedral, if and only if $\alpha \cdot R > 0$ for a generator $R$ of each extremal ray.

Given a smooth projective variety $\X$ with a strict normal crossings (SNC) divisor $\Delta = \bigcup_i \Delta_i$, we write $X = (\X, \mathcal{M}_{\X})$ for
the associated log scheme, whose log structure is 
$\mathcal{M}_{\X}(U) = \{ s \in \OO_X(U) \mid s|_{U \setminus \Delta} \in \OO_X^\times(U \setminus \Delta)\}$. We denote by $T_X$ the \emph{logarithmic} 
tangent bundle, the subsheaf of the classical tangent bundle $T_{\X}$ consisting of vector fields tangent to $\Delta$. Thus
\begin{equation}
\label{eq:log-tangent}
0 \longrightarrow T_{X} \longrightarrow T_{\X} \longrightarrow
\bigoplus_i (\iota_{\Delta_i})_* N_{\Delta_i/\X} \longrightarrow 0 ,
\qquad
\det T_{X} \cong \OO_{\X}(-K_{\X} - \Delta),
\end{equation}
where $\iota_{\Delta_i} : \Delta_i \hookrightarrow \X$ is the inclusion and $(\iota_{\Delta_i})_*$ is the pushforward (extension by zero). 

\subsection{Campana orbifolds}
\label{ss:orbifolds}
 
\begin{definition}[{\cite[Section~1.2.1]{campana2004orbifolds}}]
\label{def:orbifold}
Let $\X$ be a smooth projective variety with an SNC divisor $\Delta = \bigcup_i \Delta_i$ and let $X$ be the associated log scheme. Assign to each 
irreducible component $\Delta_i$ a weight $$\epsilon_i = 1 - \frac{1}{m_i}, \qquad m_i \in \Z_{\geq 1},$$
and set $\Delta_\epsilon = \sum_i \epsilon_i \Delta_i$. The pair $(\X, \Delta_\epsilon)$ is a
\emph{klt Campana orbifold}. It is a \emph{klt Fano orbifold} if $-(K_{\X} + \Delta_\epsilon)$ is ample.
\end{definition}

\subsection{Campana curves}
\label{ss:campana-curves}
 
Let $S$ be a geometric log point with the trivial log structure and let $(\pi : C \to S,\, f : C \to X)$ be a stable log map with the canonical log 
structure. Such an $f$ is called \emph{non-degenerate}, then $\underline{C}$ is smooth and irreducible, $f(\underline{C}) \not\subset \Delta$,
and $f^{-1}(\Delta)$ consists of marked points. For a marking $p_k$ the \emph{contact order} $\mathbf{c}_k = (c_{k,i})_i$ records the local 
multiplicity of $f^*\Delta_i$ at $p_k$, (see \cite[Section~3.2]{chen2014stable} and \cite[Section~4.1]{abramovich2014stable} for the definition in 
general, including at nodes). 

\begin{definition}[{\cite[Definition~1.1]{chen2024campana}}]
    \label{def:campana-curve}
    A non-degenerate stable log map $f : C \to X$ is a \emph{Campana curve} for $(\X, \Delta_\epsilon)$ if for every marking $p_k$ and every $i$ 
    we have either $c_{k,i} = 0$ or $c_{k,i} \geq m_i$. A Campana curve is \emph{rational} if $\underline{C}$ has genus $0$.
\end{definition}

Since logarithmic deformation preserve contact orders, the theory of moduli of stable log maps is useful to study deformation theory of Campana curves. 
Hence we work in the category of log schemes instead of the notion of a curve meeting $\Delta_i$ with specific contact orders.   

\begin{definition}[{\cite[Definition~1.2 and Definition~5.6]{chen2024campana}}]
    \label{def:campana-uniruled}
    A klt Campana orbifold $(\X,\Delta_\epsilon)$ is called \emph{Campana uniruled} if there is a dominant family of rational Campana curves whose underlying 
    curves have a non-trivial numerical class. It is called \emph{Campana rationally connected} if there is a family of rational Campana curves dominating 
    $\X$ that pass through two general points of $\X$.
\end{definition}

\subsection{The normal complex and free log curve}
\label{ss:freeness}
 
We assume that $\operatorname{char} k = 0$. For a log map $f : C \to X$ over a geometric log point $S$, the \emph{normal complex} $N_f$ is defined by
the distinguished triangle $$T_{C/S} \xrightarrow{\ df\ } f^*T_{X} \longrightarrow N_f \xrightarrow{\ [1]\ }.$$
Its cohomology controls the deformations of $f$ relative to Olsson's log stack \cite[\S~2.2.2]{chen2024campana}. In our situation, $N_f$ is a sheaf:
 
\begin{definition}[{\cite[Definition~2.5]{chen2024campana}}]
    \label{def:log-immersion}
    A log map $f : C \to X$ over a geometric log point is a \emph{log immersion} if every node and every marking of $C$ has non-zero contact order
    $\mathbf{c}$ with $\operatorname{char} k \nmid \mathbf{c}$, and if $\underline f$ is an immersion away from nodes and markings.
\end{definition}

In characteristic $0$ the divisibility condition is trivial \cite[Notation~2.3]{chen2024campana}, so the first condition says only that all nodes and
markings map into $\Delta$ with non-zero contact order. If $f$ is a log immersion then $df : T_{C/S} \to f^*T_{X}$ is the inclusion of a subbundle, 
and $N_f = \operatorname{Cok}(df)$ is a vector bundle on $\underline{C}$ \cite[Lemma~2.6]{chen2024campana}. If $f$ has in addition finitely many markings of
contact order $0$, then $N_f$ is a sheaf whose torsion free part is the normal bundle of the associated log map with only contact markings, and whose 
torsion part is supported on those markings \cite[Corollary~2.7]{chen2024campana}. 

\begin{definition}[{\cite[Definition~4.2]{chen2024campana}}]
    \label{def:free}
    A non-degenerate rational log curve $f : C \to X$ with contact orders $\varsigma$ is \emph{$\varsigma$-free} (resp.\ \emph{$\varsigma$-very free}) 
    if $H^1(N_f(-1)) = 0$ (resp.\ $H^1(N_f(-2)) = 0$). A \emph{free Campana curve} is a Campana curve which is $\varsigma$-free for its collection of 
    contact orders.
\end{definition}

\begin{proposition}[{\cite[Proposition~4.4]{chen2024campana}}]
    \label{prop:free-equiv}
    Let $\varsigma$ be a collection of non-zero contact orders, not divisible by $\operatorname{char} k$. Then $X$ is separably $\varsigma$-uniruled
    (resp.\ separably $\varsigma$-rationally connected) if and only if it admits a $\varsigma$-free (resp.\ $\varsigma$-very free) rational log curve. 
    In particular, in characteristic $0$, Campana uniruledness and Campana rational connectedness are equivalent to the existence of free and of very free 
    Campana curves respectively.
\end{proposition}

\begin{definition}[{\cite[Definition~5.9]{chen2024campana}}]
    \label{def:strongly-uniruled}
    A klt Campana orbifold $(\X, \Delta_\epsilon)$ is \emph{strongly Campana uniruled} if there exists a free Campana curve $f : C \to X$ such that
    $f_*[C]$ lies in the interior of the nef cone of curves $\Nef_1(X)$.
\end{definition}

When $\rho(X) = 1$ the interiority condition is vacuous, so Definition \ref{def:strongly-uniruled} is a strengthening of Campana uniruledness only for
higher Picard rank surfaces which is precisely the situation for del Pezzo surfaces of degree $\leq 8$. 

\subsection{Gluing and smoothing}
\label{ss:gluing}
 
We discuss the form in which we use the gluing construction of \cite[Section~2.6]{chen2024campana} and the smoothing of \cite[Section~2.7]{chen2024campana}. 
Recall from \cite[\S~2.6.1]{chen2024campana} that two log maps meeting at a point of $X^\circ = X \setminus \Delta$ can be glued to a stable log map 
over a geometric log point whose restriction to each component is the given map and whose markings and contact orders are inherited, the Campana 
condition of Definition \ref{def:campana-curve} is preserved, since the new node lies over $X^\circ$.

\begin{proposition}[{\cite[\S~2.6.1 and \S~2.7]{chen2024campana}}]
    \label{prop:gluing}
    Let $f_j : C_j \to X$, $1 \leq j \leq s$, be free Campana curves over the standard log point, and let $x_1, \dots, x_{s-1}$ be pairwise distinct 
    closed points of $\X \setminus \Delta$ with $x_j \in f_j(C_j) \cap f_{j+1}(C_{j+1})$. Then there exists a free Campana curve $f : C \to X$ with
    $$f_*[C] \;=\; \sum_{j=1}^{s} f_{j*}[C_j].$$
\end{proposition}

Adding a marking of contact order $0$ to a free Campana curve changes neither the Campana condition nor freeness: by \cite[Corollary~2.7]{chen2024campana} 
it only enlarges the torsion subsheaf of $N_f$, and torsion does not contribute to $H^1$. We use this without further comment when choosing the points 
$x_j$ above. The proof of Proposition \ref{prop:gluing} is recalled in Lemma \ref{lem:gluing}.

\subsection{Weak approximation for Campana sections}
\label{ss:WA}
 
Let $B$ be a smooth projective curve over $k$. A \emph{klt Campana fibration} is a klt Campana orbifold $(\cX,\Delta_\epsilon)$ together with a log 
fibration $\pi : (\cX,\Delta) \to B$ as in \cite[Section~3]{chen2024campana}, a \emph{Campana section} is a section satisfying the Campana condition, 
and a \emph{Campana $n$-th jet} is an admissible jet whose local multiplicities along the boundary satisfy the Campana condition 
\cite[Definition~6.2]{chen2024campana}. The following two statements are the reason for proving Theorem \ref{thmA}.

\begin{theorem}[{\cite[Theorem~6.4]{chen2024campana}}]
    \label{thm:WA}
    Assume $k$ is algebraically closed of characteristic $0$. Let $\pi : (\cX,\Delta_\epsilon) \to B$ be a klt Campana fibration such that a general 
    fibre of $\pi$ is rationally connected and strongly Campana uniruled. Let $S$ be a finite set of closed points of $B$ such that 
    $\underline\pi : \underline{\cX} \to B$ satisfies weak approximation outside $S$. Fix points $p_1,\dots,p_r \in \cX$ on distinct fibres outside $S$ 
    and, for each $j$, a Campana $n_j$-th jet $\sigma_j$ at $p_j$. Then there is a Campana section approximating the jet data $\{p_j,\sigma_j\}$.
\end{theorem}

\begin{corollary}[{\cite[Corollary~6.7]{chen2024campana}}]
    \label{cor:CRC}
    Assume $k$ is algebraically closed of characteristic $0$. Let $(\X,\Delta_\epsilon)$ be a klt Campana orbifold such that $\X$ is rationally connected 
    and $(\X,\Delta_\epsilon)$ is strongly Campana uniruled. Then $(\X,\Delta_\epsilon)$ is Campana rationally connected.
\end{corollary}

\subsection{Campana curves on orbifold}
\label{ss:orbifold-curves}
 
Let $\Delta = \{p_1,\dots,p_n\} \subset \PP^1$ with weights $m_1,\dots,m_n$, so that $(\PP^1,\Delta_\epsilon)$ is a klt Fano orbifold when
$$\sum_{k=1}^n \Bigl(1 - \frac{1}{m_k}\Bigr) \;=\; \operatorname{deg} \Delta_\epsilon \;<\; 2 \;=\; \operatorname{deg}(-K_{\PP^1}),$$
that is, precisely when the orbifold Euler characteristic $2 - \sum_k (1-1/m_k)$ is positive. By \cite[Theorem~7.1]{chen2024campana} such an orbifold 
contains free Campana curves, obtained by composing the quotient map by a suitable finite subgroup of $PGL_2(k)$ (a Belyi map with the prescribed branch 
data) with a general cover of $\PP^1$ of suitable degree. Since $\rho(\PP^1) = 1$, freeness already gives 
strong Campana uniruledness. Lemma \ref{lem:covers} below is the variant of the numerical part of this statement that we use.

\subsection{Del Pezzo surfaces}
\label{ss:dP}
 
Let $\X$ be a del Pezzo surface of degree $d = K_{\X}^2$.
 
\begin{lemma}
\label{lem:dP-model}
For $1 \leq d \leq 7$, $\X$ is the blow up $\beta : \X \to \PP^2$ of $\PP^2$
at $9-d$ points in general position, and $$K_{\X} = -3H + \sum_{i=1}^{9-d} E_i ,$$ where $H = \beta^*\OO_{\PP^2}(1)$ and $E_1, \dots, E_{9-d}$ are the 
exceptional curves. For $d = 8$, either $\X$ is the blow up of $\PP^2$ at one point, with $K_{\X} = -3H + E$, or $\X = \PP^1 \times \PP^1$, with 
$K_{\X} = -2H_1 - 2H_2$ for the two rulings $H_i, i = 1,2$. For $d = 9$, $X = \PP^2$.
\end{lemma}

\begin{lemma}
    \label{lem:dP-cones}
    The pseudo-effective cone $\overline{\operatorname{Eff}}(X)$ is rational polyhedral, generated by
    \begin{itemize}
        \item the classes of the $(-1)$-curves, if $1 \leq d \leq 7$;
        \item $E$ and $H-E$, if $X$ is the blow up of $\PP^2$ at one point;
        \item $H_1$ and $H_2$, if $X = \PP^1 \times \PP^1$;
        \item $H$, if $X = \PP^2$.
    \end{itemize}
    Consequently a class $\alpha \in N_1(X)_{\R}$ lies in the interior of $\Nef_1(\X)$ if and only if $\alpha$ has strictly positive intersection 
    with each of the generators listed above, for $1 \leq d \leq 7$, $\alpha \cdot E > 0$ for every $(-1)$-curve $E$.
\end{lemma}

\begin{lemma}
    \label{lem:conic-class}
    Let $F$ be a conic class on a del Pezzo surface $\X$. Then:
    \begin{enumerate}
        \item $F$ is nef.
        \item $h^0(\X,F) = 2$ and $h^i(\X,F) = 0$ for $i > 0$.
        \item $|F|$ is a base point free pencil and the induced morphism $\varphi_{|F|} : \X \to \PP^1$ is a conic bundle whose general fibre is a
        smooth rational curve.
    \end{enumerate}
\end{lemma}

\begin{proof}
    (1) Suppose $F \cdot C < 0$ for an irreducible curve $C$. Then $C \subseteq \operatorname{Supp} F$ and $C^2 < 0$, so $C$ is a $(-1)$-curve. 
    Write $F = aC + F'$ with $a \geq 1$ and $F'$ effective not containing $C$. From $2 = -K_{\X} \cdot F = a + (-K_{\X} \cdot F')$ and 
    $-K_{\X} \cdot F' \geq 0$ we get $a \leq 2$. If $a = 2$ then $-K_{\X} \cdot F' = 0$, so $F' = 0$ and $F^2 = 4C^2 = -4 \neq 0$, a contradiction. 
    If $a = 1$ then $-K_{\X} \cdot F' = 1$, so when $\X$ is a del Pezzo surface of degree $d \neq 1$, $F'$ is a $(-1)$-curve $C'$. 
    From $0 = F^2 = -1 + 2 C\cdot C' - 1$ we get $C \cdot C' = 1$ and hence $F \cdot C = -1 + 1 = 0$, again a contradiction. 
    When $\X$ is a del Pezzo surface of degree $1$, an irreducible curve with $-K_{\X} \cdot F' = 1$ can be a curve $C' \in |-K_{\X}|$ in the anticanonical class with 
    $C'^2 = 1$. Hence $C \cdot C' = 1$ and $F^2 = C^2 + 2C \cdot C' + C'^2 = 2 \neq 0$, again a contradiction.  
     
    (2) By Riemann-Roch theorem $\chi(\OO_{\X}(F)) = 1 + \tfrac12 F\cdot(F - K_{\X}) = 2$, and $h^i(\X,F) = h^i(\X, K_{\X} + (F - K_{\X})) = 0$ 
    for $i > 0$ by Kodaira vanishing, since $F - K_{\X}$ is ample by (1).
     
    (3) Write $|F| = |M| + Z$ with $Z$ the fixed part. As $F$ is nef and $F^2 = 0$ we get $F \cdot M = F \cdot Z = 0$, hence 
    $M^2 = M\cdot F - M \cdot Z = -M\cdot Z \leq 0$, since $M$ is movable part, $M^2 \geq 0$, so $M^2 = 0$ and $M \cdot Z = 0$. If $Z \neq 0$ then 
    $-K_{\X} \cdot Z \geq 1$ and $-K_{\X} \cdot M \geq 1$, so $-K_{\X} \cdot M = 1$ and $p_a(M) = 1 + \tfrac12 (M^2 + K_{\X}\cdot M) = \tfrac12$, 
    which is a contradiction. Hence $Z = 0$. Two general members of the pencil $|F|$ meet only along the base locus, and $F^2 = 0$, so $|F|$ is base 
    point free. Finally $p_a(F) = 1 + \tfrac12(F^2 + K_{\X} \cdot F) = 0$, so by generic smoothness the general member of $|F|$ is a smooth rational curve.
\end{proof}

Finally we record the case $d = 9$ of Theorem \ref{thmD}, which appears in a unpublished notes by Chen, Tanimoto and Lehmann. It is imediate because 
$\operatorname{Pic}(\PP^2) = \Z H$.
 
\begin{proposition}
    \label{prop:P2}
    Let $\X = \PP^2$, let $D$ be a smooth irreducible curve of degree $t$ and let
    $\epsilon = 1-1/m$ with $m \geq 2$. Then $-(K_{\X} + \epsilon D) = (3 -\epsilon t) H$, and
    therefore:
    \begin{enumerate}
        \item if $t \leq 3$, then $-(K_{\X} + \epsilon D)$ is ample for every $m \geq 2$.
        \item if $t = 4$, then $-(K_{\X} + \epsilon D)$ is ample exactly for $\epsilon = \tfrac12, \tfrac23$.
        \item if $t = 5$, then $-(K_{\X} + \epsilon D)$ is ample exactly for $\epsilon = \tfrac12$.
        \item if $t = 6$, then $-(K_{\X} + \epsilon D)$ is nef, and trivial, exactly for $\epsilon = \tfrac12$, and is never ample.
        \item if $t \geq 7$, then $-(K_{\X} + \epsilon D)$ is never nef.
    \end{enumerate}
\end{proposition}

\section{Campana Rational Connectedness}\label{sec:mainthm}

Throughout this section $k$ is an algebraically closed field of characteristic $0$. We keep the notation of Section 2: $(X, \Delta_{\epsilon})$ is a klt Campana orbifold,
$\X$ denotes the underlying smooth projective surface, $\Delta = \sum_i \Delta_i$ is the reduced SNC boundary, $\epsilon_i = 1 - 1/m_i$ with $m_i \in \Z_{\geq 1}$, and 
$X$ is the associated log scheme. We write $T_X$ for the logarithmic tangent bundle and $T_{\X}$ for the classical one.

The proof of the main theorems is organised as follows. In \S\ref{ss:lemmas} we record two elementary facts that control the construction of a free Campana curve: 
ramified covers of $\PP^1$ (Lemma~\ref{lem:covers}), the degree of the normal bundle of a Campana curve obtained by base change (Lemma~\ref{lem:degree}).
In \S\ref{ss:construction} we produce Campana curves from the conic bundle structures on $\X$; we show that the ampleness of $-(K_{\X} + \Delta_{\epsilon})$ is exactly the inequality 
required for a suitable cover of $\PP^1$ to exist. The main theorems are proved in \S\ref{ss:proofs}. 

\subsection{Lemmas}\label{ss:lemmas}

\begin{lemma}\label{lem:covers}
    Let $p_1, \dots, p_n \in \PP^1$ be distinct closed points and let $w_1, \dots, w_n \in \Z_{\geq 1}$. Consider finite morphisms $g : \PP^1 \to \PP^1$
    such that \begin{equation}\label{eq:ram}
        e_q \geq w_k \qquad \text{for every } 1\leq k\leq n \text{ and every } q \in g^{-1}(p_k),
    \end{equation}
    where $e_q$ denotes the ramification index of $g$ at $q$. Then:
    \begin{enumerate}
        \item such a $g$ exists if and only if $\displaystyle\sum_{k=1}^{n}\Bigl(1-\frac{1}{w_k}\Bigr) < 2$;
        \item if moreover $n \neq 1$, then $g$ can be chosen to be \'etale over $\PP^1 \setminus \{p_1, \dots, p_n\}$.
    \end{enumerate}
\end{lemma}

\begin{proof}
    Let $g$ be of degree $N$ and satisfy \eqref{eq:ram}, and put $n_k = \#g^{-1}(p_k)$. Then $n_k \leq N/w_k$, and the Riemann-Hurwitz formula
    gives $$ 2N-2 \;=\; \sum_{q}(e_q-1) \;\geq\; \sum_{k = 1}^n \bigl(N-n_k\bigr) \;\geq\; N\sum_{k = 1}^n \Bigl(1 - \frac{1}{w_k}\Bigr),$$
    whence $\sum_k (1 - 1/w_k) \leq 2 - 2/N <2$.

    Conversely, assume $\sum_k(1 - 1/w_k) < 2$. Without loss of generality we can assume $w_k\geq 2$ for all $k$; then $n \leq 3$, since $n \geq 4$ would imply 
    $\sum_k (1 - 1/w_k) \geq n/2 \geq 2$. If $n = 0$ take $g = \mathrm{id}$. If $n = 1$ take 
    $g(z) = z^{w_1}$ with $p_1 = 0$. If $n = 2$ take $g(z) = z^w$ with $w = \operatorname{lcm}(w_1, w_2)$ and $p_1 = 0$, $p_2 = \infty$; 
    this $g$ is \'etale outside $\{0,\infty\}$. If $n = 3$, the inequality gives $1/w_1 + 1/w_2 + 1/w_3 > 1$, whose solutions are, up to permutation,
    $$(2,2,a)\ (a\geq 2), \qquad (2,3,3), \qquad (2,3,4), \qquad (2,3,5).$$
    Let $G \subset \mathrm{PGL}_2(k)$ be the corresponding finite subgroup (dihedral of order $2a$, $\mathfrak{A}_4$, $\mathfrak{S}_4$, $\mathfrak{A}_5$) 
    and let $g_0 : \PP^1 \to \PP^1 \cong \PP^1/G$ be the quotient map. Then $g_0$ is \'etale outside three points, over which the ramification indices are constant equal to
    $w_1,w_2,w_3$ respectively; composing with an automorphism of the image which carries the three branch points to $p_1,p_2,p_3$ gives the required $g$.
\end{proof}

\begin{remark}\label{rem:no-total-ramification}
    If $w_k = m$ for all $k$, By Lemma~\ref{lem:covers} a cover of $\PP^1$ with all contact orders at $p_1, \dots, p_n$ at least $m$ exists 
    if and only if $n\, \epsilon < 2$, where $\epsilon = 1 - 1/m$; in particular $n \leq 2$ for $m \geq 3$, and $n \leq 3$ for $m = 2$.  
    For $w_k = m$ and $e_q = m$ (total ramification) the bound is $n \leq 2$. No choice of the degree of $g$ improves this since 
    the domain of a cover with $n\,\epsilon \geq 2$ has positive genus. Hence our construction starts from curves meeting the boundary in a controlled
    number of points.
\end{remark}

Recall \cite[Definition~2.5]{chen2024campana} that, a log map $f: C\to X$ over a geometric log point is a \emph{log immersion} if every node 
and every marking of $C$ has contact order $\mathbf{c}\neq 0$ with $\operatorname{char}\mathbf{k} \nmid \mathbf{c}$, and 
if $\underline{f}$ is an immersion away from the nodes and markings. In characteristic $0$ the divisibility requirement is vacuous by
\cite[Notation~2.3]{chen2024campana}, so the first condition simply says that all nodes and markings map into $\Delta$ with non-zero contact order.  
By \cite[Lemma~2.6]{chen2024campana}, if $f$ is a log immersion then $df\colon T_{C/S}\to f^*T_X$ is the inclusion of a subbundle, so that
$N_f=\operatorname{Cok}(df)$ is a vector bundle on $C$.

\begin{lemma}\label{lem:degree}
    Let $(X,\Delta_\epsilon)$ be a klt Campana orbifold with $\dim\X=2$ and let
    $C'\subset\X$ be a smooth rational curve with $C'\not\subset\Delta$ and
    $C'\cap\operatorname{Sing}(\Delta)=\emptyset$.  Denote by
    $h\colon\underline{C}'\cong\PP^1\hookrightarrow\X$ the inclusion, write
    \[
    h^*\Delta \;=\; \sum_{k=1}^{n}\mu_k\,p'_k ,\qquad
    \mu_k\geq1,\quad \sum_{k=1}^{n}\mu_k=\Delta\cdot C' ,
    \]
    with $p'_1,\dots,p'_n$ distinct, and let $\Delta_{i(k)}$ be the unique component
    of $\Delta$ through $h(p'_k)$.  Let $g\colon\underline{C}\cong\PP^1\to
    \underline{C}'$ be a finite morphism of degree $N$ such that
    \begin{equation}\label{eq:contact}
    e_q\,\mu_k \;\geq\; m_{i(k)}\qquad\text{for every $k$ and every }
    q\in g^{-1}(p'_k).
    \end{equation}
    Regard $C'$ as a log curve over $S$ by means of the markings
    $p'_1,\dots,p'_n$, and $C$ as a log curve over $S$ by means of the markings
    $g^{-1}(\{p'_1,\dots,p'_n\})$ together with a (possibly empty) finite set $P$ of
    further points whose images avoid $\Delta$.  Set $f:=h\circ g\colon C\to X$ and
    \[
    \delta \;:=\; N\Bigl[\bigl(-K_{\X}-\Delta\bigr)\cdot C'+n-2\Bigr].
    \]
    Then:
    \begin{enumerate}
        \item $f$ is a Campana curve for $(X,\Delta_\epsilon)$, with contact order
        $c_q=e_q\mu_k$ at $q\in g^{-1}(p'_k)$ and $c_q=0$ at the markings in $P$;
        moreover $f_*[C]=N[C']$;
        \item $h\colon C'\to X$ is a log immersion, and $N_h$ is a line bundle on
        $\PP^1$ of degree $(-K_{\X}-\Delta)\cdot C'+n-2$;
        \item $df$ is injective, $N_f=\operatorname{Cok}(df)$, and there is an exact sequence
        \[
        0\longrightarrow N_g\longrightarrow N_f\longrightarrow
        g^*N_h\longrightarrow 0,\qquad N_g:=\operatorname{Cok}(dg),
        \]
        in which $N_g$ is a torsion sheaf; in particular the torsion subsheaf of
        $N_f$ is $N_g$ and its torsion-free quotient is the line bundle
        $g^*N_h\cong\OO_{\PP^1}(\delta)$;
        \item $H^1(N_f(-1))=0$ if and only if $\delta\geq0$, i.e.\ if and only if
          \[
             \bigl(-K_{\X}-\Delta\bigr)\cdot C'+n\;\geq\;2 .
          \]
    \end{enumerate}
    If moreover $P=\emptyset$ and $g$ is \'etale over
    $\underline{C}' \setminus \{p'_1,\dots,p'_n\}$, then $N_g=0$ and
    $N_f\cong\OO_{\PP^1}(\delta)$ is a line bundle.
\end{lemma}

\begin{proof}
    (1)  Let $q\in g^{-1}(p'_k)$.  The multiplicity of $f^*\Delta_{i(k)}$ at $q$ is
    $e_q\mu_k$, which is $\geq m_{i(k)}$ by \eqref{eq:contact}, and the
    multiplicities of $f^*\Delta_i$ at $q$ vanish for $i\neq i(k)$ because
    $h(p'_k)\notin\operatorname{Sing}(\Delta)$; at a marking in $P$ all
    multiplicities vanish.  Hence $f$ is a Campana curve in the sense of
    Definition~\ref{def:campana-curve}.  Since $h$ is a closed immersion,
    so $f_*[C]=\deg(g)\cdot[C']=N[C']$.
     
    (2)  We check \cite[Definition~2.5]{chen2024campana} for $h$.  The domain
    $\underline{C}'$ is smooth, so $h$ has no nodes; its markings are
    $p'_1,\dots,p'_n$, and the contact order of $h$ at $p'_k$ is $\mu_k\geq1$ along
    $\Delta_{i(k)}$ and $0$ along the other components, hence non-zero, and the
    divisibility condition is vacuous as $\operatorname{char}\mathbf{k}=0$.  Finally
    $\underline{h}$ is a closed immersion, hence an immersion everywhere.  So $h$ is
    a log immersion, and by \cite[Lemma~2.6]{chen2024campana} the map
    $dh\colon T_{C'/S}\to h^*T_X$ is the inclusion of a subbundle and
    $N_h=\operatorname{Cok}(dh)$ is a vector bundle of rank
    $\operatorname{rk}(h^*T_X)-1=1$, i.e.\ a line bundle on $\PP^1$.
     
    For its degree, observe that $\deg h^*T_X=\deg\det h^*T_X=\deg h^*\det T_X$,
    which by the exact sequence \eqref{eq:log-tangent} equals $(-K_{\X}-\Delta)\cdot C'$.  Since
    $C'\to S$ is a log curve with $n$ markings over the standard log point,
    $\Omega^1_{C'/S}=\omega_{\underline{C}'}(p'_1+\dots+p'_n)$ and therefore
    $T_{C'/S}\cong\OO_{\PP^1}(2-n)$.  Hence
    \[
    \deg N_h \;=\; \deg h^*T_X-\deg T_{C'/S}
    \;=\;\bigl(-K_{\X}-\Delta\bigr)\cdot C'-(2-n).
    \]
     
    (3)  Every point of $g^{-1}(\{p'_1,\dots,p'_n\})$ is a marking of $C$, so the
    log structure of $C$ is the pullback of the log structure of $C'$ and $g$
    is a morphism of log curves over $S$. Let $dg\colon T_{C/S}\to g^*T_{C'/S}$ be
    the induced map.  Both sheaves are line bundles on $\PP^1$, and $dg$ is non-zero
    because $g$ is separable, hence generically log \'etale. Therefore $dg$ is
    injective with torsion cokernel $N_g$.  As $g$ is finite flat, applying $g^*$ to
    $0\to T_{C'/S}\to h^*T_X\to N_h\to0$ gives the exact sequence
    \[
    0\longrightarrow g^*T_{C'/S}\longrightarrow f^*T_X\longrightarrow
    g^*N_h\longrightarrow0 ,
    \]
    so $g^*T_{C'/S}$ is a subbundle of $f^*T_X$.  By functoriality of the
    logarithmic differential, $df=g^*(dh)\circ dg$. In particular $df$ is injective,
    the normal complex $N_f$ is the sheaf $\operatorname{Cok}(df)=f^*T_X/T_{C/S}$, and we have
    inclusions of $\OO_C$-modules
    \[
    T_{C/S}\;\subseteq\;g^*T_{C'/S}\;\subseteq\;f^*T_X .
    \]
    The associated exact sequence
    $0\to g^*T_{C'/S}/T_{C/S}\to f^*T_X/T_{C/S}\to f^*T_X/g^*T_{C'/S}\to0$
    is the asserted one.  Since $g^*N_h$ is locally free and $N_g$ is torsion, $N_g$
    is the torsion subsheaf of $N_f$ and $g^*N_h$ is its torsion-free quotient. Its
    degree is $N\deg N_h=\delta$ by (2), so $g^*N_h\cong\OO_{\PP^1}(\delta)$.
     
    (4)  Twisting the exact sequence of (3) by $\OO_{\PP^1}(-1)$ and by vanishing of
    $H^1$ of a torsion sheaf and $H^2$ of a sheaf on a curve, we
    get $H^1(N_f(-1))\cong H^1(\OO_{\PP^1}(\delta-1))$, which vanishes if and only
    if $\delta\geq0$.  As $N\geq1$, this is equivalent to
    $(-K_{\X}-\Delta)\cdot C'+n\geq2$.
     
    Finally assume $P=\emptyset$ and $g$ \'etale outside the marked fibres, and put
    $r=\sum_kn_k$ with $n_k=\#g^{-1}(p'_k)$, so that $T_{C/S}\cong\OO_{\PP^1}(2-r)$.
    By Riemann--Hurwitz theorem $2N-2=\sum_k(N-n_k)=nN-r$, i.e.\ $2-r=N(2-n)$, so
    $\deg T_{C/S}=\deg g^*T_{C'/S}$ and the injection $dg$ of line bundles is an
    isomorphism.  Hence $N_g=0$ and $N_f\cong g^*N_h\cong\OO_{\PP^1}(\delta)$.
\end{proof}

\begin{remark}\label{rem:transverse}
    If $C'$ meets $\Delta$ transversally, i.e.\ $\mu_k=1$ for all $k$ and
    $n=\Delta\cdot C'$, then the condition in Lemma~\ref{lem:degree}(4) reads
    $-K_{\X}\cdot C'\geq2$, that is $(C')^2\geq0$: there is no condition on
    $\Delta\cdot C'$ at all.  Tangency to $\Delta$ relaxes the hypothesis of
    Lemma~\ref{lem:covers} but tightens the one of Lemma~\ref{lem:degree}; the
    constructions below are arranged so that both hold.
\end{remark}
     
\begin{remark}\label{rem:markings}
    Adding a marking with contact order $0$ changes neither $\delta$ nor the
    freeness of $f$: the quantity $\delta$ of Lemma~\ref{lem:degree} does not depend
    on $P$, and by part (3) the extra marking only enlarges the torsion subsheaf
    $N_g$.  This will be used in Lemma~\ref{lem:gluing}.
\end{remark}
     
\subsection{Free Campana curves on del Pezzo surfaces}\label{ss:construction}
     
Let $\X$ be a del Pezzo surface.  A class $F\in\operatorname{Pic}(\X)$ is a
\emph{conic class} if $F$ is effective, $F^2=0$ and $-K_{\X}\cdot F=2$.  For 
such an $F$ one has $\chi(\OO_{\X}(F))=2$ by Riemann--Roch and
$h^1=h^2=0$ by Kawamata--Viehweg vanishing (as $F-K_{\X}$ is ample), so
$h^0(\X,F)=2$; the system $|F|$ is base point free and defines a conic bundle
$\X\to\PP^1$, and $p_a(F)=1+\frac12(F^2+K_{\X}\cdot F)=0$, so the general member
of $|F|$ is a smooth rational curve.
     
\begin{proposition}\label{prop:conic}
    Let $(X,\Delta_\epsilon)$ be a klt Campana orbifold whose underlying surface
    $\X$ is a del Pezzo surface, and assume that $-(K_{\X}+\Delta_\epsilon)$ is
    ample. Let $F$ be a conic class on $\X$.  Then there exist $N\geq1$ and a free
    Campana curve $f\colon C\to X$ with $f_*[C]=N[F]$, that is, its image is a
    general member of $|F|$.  If $\Delta\cdot F\neq1$, we further get
    $N_f\cong\OO_{\PP^1}$.
\end{proposition}
     
\begin{proof}
    Let $C'\in|F|$ be a general member.  Since $|F|$ is base point free of
    projective dimension $1$, By Bertini's theorem and generic smoothness in
    characteristic $0$ we take $C'$ is a smooth rational curve, that
    $C'\not\subset\Delta$, that $C'$ avoids the finite set
    $\operatorname{Sing}(\Delta)$, and that $C'$ meets $\Delta$ transversally in
    $n=\Delta\cdot F$ distinct points.  Thus $\mu_k=1$ for every $k$ in the notation
    of Lemma~\ref{lem:degree}, so \eqref{eq:contact} amounts to $e_q\geq w_k$ with
    $w_k=m_{i(k)}$. By ampleness of $-(K_{\X}+\Delta_\epsilon)$ we have
    \begin{equation}\label{eq:key}
    \sum_{k=1}^{n}\Bigl(1-\frac{1}{w_k}\Bigr)
    \;=\;\sum_i\epsilon_i\,(\Delta_i\cdot F)
    \;=\;\Delta_\epsilon\cdot F\;<\;-K_{\X}\cdot F\;=\;2 ,
    \end{equation}
    so by Lemma~\ref{lem:covers} we have a finite morphism $g$ of some degree $N$
    satisfying \eqref{eq:contact}, \'etale outside the marked fibres when
    $n\neq1$.  Finally
    \[
    \bigl(-K_{\X}-\Delta\bigr)\cdot C'+n-2
    \;=\;\bigl(2-\Delta\cdot F\bigr)+\Delta\cdot F-2\;=\;0 ,
    \]
    so Lemma~\ref{lem:degree} applies with $\delta=0$. Hence the map $f = h \circ g$ is a
    free Campana curve with $f_*[C]=N[F]$, and $N_f\cong\OO_{\PP^1}$ when
    $n=\Delta\cdot F\neq1$.
\end{proof}
     
    Inequality \eqref{eq:key} says that for a conic class the
    ampleness of $-(K_{\X}+\Delta_\epsilon)$ gives precisely that the orbifold
    $\PP^1$ obtained by weighting the points of $C'\cap\Delta$ has positive orbifold
    Euler characteristic, which by Lemma~\ref{lem:covers} is exactly the condition
    for the required cover to exist.
     
    Conic classes do not suffice when $\rho(\X)\leq2$ and $\X\neq\PP^1\times\PP^1$;
    there we use the pullback of the line class.
     
\begin{proposition}\label{prop:line}
    Let $(X,\epsilon D)$ be a klt Campana orbifold with irreducible boundary $D$ and
    $\epsilon=1-1/m$, whose underlying surface $\X$ is $\PP^2$ or the blow up of
    $\PP^2$ at one or two points, and assume $-(K_{\X}+\epsilon D)$ is ample.  Let
    $\beta\colon\X\to\PP^2$ be a birational morphism and $H=\beta^*\OO_{\PP^2}(1)$.
    Then there exist $N\geq1$ and a free Campana curve $f\colon C\to X$ with
    $f_*[C]=N[H]$.
\end{proposition}
     
\begin{proof}
    Since $\Delta=D$ is irreducible and SNC, $D$ is smooth; put $t:=D\cdot H$, so
    that $t=\deg\beta_*D$, with $t=0$ if $D$ is $\beta$-exceptional.  Ampleness
    gives
    \begin{equation}\label{eq:line-ample}
    \epsilon\,t\;=\;\epsilon D\cdot H\;<\;-K_{\X}\cdot H\;=\;3 .
    \end{equation}
     
    \emph{Case 1: $\epsilon t<2$.}  Let $C'\in|H|$ be general; as in the proof of
    Proposition~\ref{prop:conic} it is a smooth rational curve meeting $D$
    transversally in $n=t$ distinct points and avoiding the points blown up by
    $\beta$.  Here $w_k=m$ for all $k$, and $\sum_k(1-1/w_k)=\epsilon t<2$, so
    Lemma~\ref{lem:covers} applies, while
    \[
    \bigl(-K_{\X}-D\bigr)\cdot C'+n-2\;=\;(3-t)+t-2\;=\;1\;\geq\;0 ,
    \]
    so Lemma~\ref{lem:degree} produces a free Campana curve with $f_*[C]=N[H]$.
     
    \emph{Case 2: $\epsilon t\geq2$.}  Since $\epsilon < 1$ we get $t \geq 3$, so
    $B := \beta_*D$ is an irreducible plane curve of degree $t \geq 3$. In particular
    $B$ is not a line.  Moreover $\operatorname{Sing}(B)$ is contained in the set of
    points blown up by $\beta$, because $D$ is smooth and maps birationally onto
    $B$.  As $\operatorname{char}\mathbf{k} = 0$, biduality holds for $B$ and the
    general tangent line $L$ of $B$ is a simple tangent: it is tangent at a single
    smooth point of $B$ with contact order exactly $2$ and meets $B$ transversally
    at $t-2$ further points.  
    Furthermore $L$ avoids the at most two points blown up by $\beta$: otherwise all
    tangent lines of $B$ would pass through a fixed point, forcing $B$ to be a
    strange curve, hence a line in characteristic $0$.  Consequently the strict
    transform $C'$ of $L$ equals $\beta^*L$, lies in $|H|$, is a smooth rational
    curve, and
    \[
    h^*D \;=\; 2p'_1+p'_2+\dots+p'_{t-1},\qquad n=t-1 ,
    \]
    so that $w_1=\lceil m/2\rceil$ and $w_k=m$ for $2\leq k\leq t-1$.  We claim
    \begin{equation}\label{eq:tangent-check}
    \Bigl(1-\frac{1}{\lceil m/2\rceil}\Bigr)+(t-2)\,\epsilon\;<\;2 .
    \end{equation}
    If $m=2$ then $\lceil m/2\rceil=1$ and \eqref{eq:line-ample} gives $t\leq5$, so
    the left hand side equals $(t-2)/2\leq3/2$.  If $m=3$ then
    \eqref{eq:line-ample} gives $t\leq4$ and the left hand side is at most
    $\frac12+\frac23\cdot2=\frac{11}{6}$.  If $m\geq4$ then $\epsilon\geq\frac34$
    and \eqref{eq:line-ample} gives $t<4$, so the left hand side is
    $(1-1/\lceil m/2\rceil)+\epsilon<1+1=2$.  This proves
    \eqref{eq:tangent-check}, so Lemma~\ref{lem:covers} applies.  Finally
    \[
    \bigl(-K_{\X}-D\bigr)\cdot C'+n-2\;=\;(3-t)+(t-1)-2\;=\;0 ,
    \]
    and Lemma~\ref{lem:degree} again yields a free Campana curve with
    $f_*[C]=N[H]$.
\end{proof}
     
\begin{lemma}\label{lem:positivity}
    Let $\X$ be a del Pezzo surface and define a finite set
    $\mathcal{G}(\X)\subset\operatorname{Pic}(\X)$ as follows:
    \begin{itemize}
    \item if $\deg\X\leq6$: all conic classes on $\X$.
    \item if $\deg\X=7$: $\{H-E_1,\;H-E_2,\;H\}$.
    \item if $\X$ is the blow up of $\PP^2$ at one point: $\{H-E,\;H\}$.
    \item if $\X=\PP^1\times\PP^1$: $\{H_1,\;H_2\}$.
    \item if $\X=\PP^2$: $\{H\}$.
    \end{itemize}
    Then:
    \begin{enumerate}
    \item every $\gamma\in\mathcal{G}(\X)$ is nef, and every $\gamma\neq H$ is a
          conic class.
    \item for every extremal ray $R$ of $\overline{\operatorname{Eff}}(\X)$ there is
          a $\gamma\in\mathcal{G}(\X)$ with $\gamma\cdot R>0$. Consequently
          $\sum_{\gamma}a_\gamma\gamma$ lies in the interior of $\Nef_1(\X)$ for any
          $a_\gamma\in\Z_{>0}$.
    \item $\gamma\cdot\gamma'>0$ for any two distinct
          $\gamma,\gamma'\in\mathcal{G}(\X)$.
    \end{enumerate}
\end{lemma}
     
\begin{proof}
    (1)  A conic class is represented by an irreducible curve of self-intersection
    $0$, hence meets every irreducible curve non-negatively and is nef. $H$, $H_1$,
    $H_2$ are pullbacks of ample classes, hence nef.  For $\deg\X=7$ we have
    $(H-E_i)^2=0$ and $-K_{\X}\cdot(H-E_i)=(3H-E_1-E_2)\cdot(H-E_i)=2$. For the blow
    up of $\PP^2$ at one point, $(H-E)^2=0$ and $-K_{\X}\cdot(H-E)=2$. For
    $\PP^1\times\PP^1$, $H_i^2=0$ and $-K_{\X}\cdot H_i=2$.  So all these are conic
    classes.
     
    (2)  Recall from Section~2 that $\overline{\operatorname{Eff}}(\X)$ is generated
    by the $(-1)$-curves when $\deg\X\leq7$, by $E$ and $H-E$ when $\X$ is the blow
    up of $\PP^2$ at one point, by $H_1$ and $H_2$ when $\X=\PP^1\times\PP^1$, and
    by $H$ when $\X=\PP^2$.
     
    Assume $\deg\X=d\leq6$ and let $E$ be a $(-1)$-curve on $\X$.  Contracting $E$
    yields a del Pezzo surface of degree $d+1\leq7$, which is a blow up of $\PP^2$.
    After composition, we obtain a birational morphism $\beta\colon\X\to\PP^2$ whose first
    step contracts $E$.  With $H=\beta^*\OO_{\PP^2}(1)$ we have $H\cdot E=0$, so
    $F:=H-E$ satisfies $F^2=0$ and $-K_{\X}\cdot F=2$. It is effective, since it is the
    class of the strict transform of a line through $\beta(E)$.  Thus $F$ is a conic
    class with $F\cdot E=1>0$.
     
    If $d=7$ the $(-1)$-curves are $E_1$, $E_2$ and $H-E_1-E_2$, and
    $(H-E_1)\cdot E_1=(H-E_2)\cdot E_2=H\cdot(H-E_1-E_2)=1$.  If $\X$ is the blow up
    of $\PP^2$ at one point, $(H-E)\cdot E=1$ and $H\cdot(H-E)=1$.  If
    $\X=\PP^1\times\PP^1$, $H_1\cdot H_2=1$.  If $\X=\PP^2$, $H^2=1$.
     
    For the consequence, let $\gamma_0=\sum_\gamma a_\gamma\gamma$ with
    $a_\gamma>0$.  For an extremal ray $R$ every summand satisfies
    $\gamma\cdot R\geq0$ by (1) and at least one satisfies $\gamma\cdot R>0$. Hence
    $\gamma_0\cdot R>0$ for every extremal ray of
    $\overline{\operatorname{Eff}}(\X)$, i.e.\ $\gamma_0$ lies in the interior of
    the dual cone $\Nef_1(\X)$.
     
    (3)  Let $F\neq F'$ be conic classes and suppose $F\cdot F'=0$.  As $F$ is nef
    with $F\cdot F'=0$, the general member of $|F'|$ is contracted by the conic
    bundle defined by $|F|$, hence is contained in a fibre.  It is not a whole fibre
    (else $F=F'$), so it is a component of a reducible fibre and therefore a
    $(-1)$-curve, contradicting $(F')^2=0$.  In the remaining cases
    $(H-E_i)\cdot H=(H-E)\cdot H=H_1\cdot H_2=1>0$.
\end{proof}
     
\begin{lemma}\label{lem:gluing}
    Let $f_j\colon C_j\to X$, $1\leq j\leq s$, be free Campana curves over the
    standard log point, and let $x_1,\dots,x_{s-1}$ be pairwise distinct closed
    points of $\X\setminus\Delta$ with $x_j\in f_j(C_j)\cap f_{j+1}(C_{j+1})$.
    Then there exists a free Campana curve $f\colon C\to X$ with
    $f_*[C]=\sum_{j=1}^{s}f_{j*}[C_j]$.
\end{lemma}
     
\begin{proof}
    We may assume $s=2$ and use induction, the general case following by gluing
    $f_1,\dots,f_j$ first and then attaching $f_{j+1}$ at $x_j$.
     
    Choose $q_1\in f_1^{-1}(x_1)$ and $q_2\in f_2^{-1}(x_1)$ and add them to the
    sets of markings of $C_1$ and $C_2$, with contact order $0$. By
    Remark~\ref{rem:markings} the maps $f_1,f_2$ remain free Campana curves.  Since
    $x_1\in\X\setminus\Delta$, the gluing construction of
    \cite[Section~2.6.1]{chen2024campana} applies and produces a stable log map
    $f^0\colon C^0\to X$ over a geometric log point, where
    $C^0=C_1\cup_{q_1=q_2}C_2$, whose restriction to $C_j$ is $f_j$ and whose
    markings are those of $f_1$ and $f_2$ other than $q_1,q_2$, with unchanged
    contact orders.  In particular $f^0$ satisfies the Campana condition of
    Definition~\ref{def:campana-curve}.  By the log smoothing of
    \cite[Section~2.7]{chen2024campana} there is a stable log map $f\colon C\to X$
    with the same discrete data as $f^0$, with smooth irreducible domain, and with
    $H^1(N_f(-1))=0$. It is therefore a free Campana curve, and
    $f_*[C]=f^0_*[C^0]=f_{1*}[C_1]+f_{2*}[C_2]$.
\end{proof}
     
\subsection{Proofs of the main theorems}\label{ss:proofs}
     
\begin{theorem}\label{thm:free-campana-curve}
    Let $(X,\epsilon D)$ be a del Pezzo orbifold with an irreducible boundary.
    There exists a free Campana curve $f\colon C\to X$ such that $f_*[C]$ lies in
    the interior of the nef cone of curves $\Nef_1(X)$.
\end{theorem}
     
\begin{proof}
    Write $\mathcal{G}(\X)=\{\gamma_1,\dots,\gamma_s\}$ as in
    Lemma~\ref{lem:positivity}.  For each $j$ we construct a free Campana curve
    $f_j\colon C_j\to X$ with $f_{j*}[C_j]=N_j\gamma_j$, $N_j\geq1$: if $\gamma_j$
    is a conic class we apply Proposition~\ref{prop:conic}, and if $\gamma_j=H$
    (which occurs only when $\X$ is $\PP^2$, the blow up of $\PP^2$ at one point, or
    a del Pezzo surface of degree $7$) we apply Proposition~\ref{prop:line}.  In
    either case the underlying curve $C'_j\subset\X$ is a general member of a
    family $\mathcal{F}_j$ of curves of class $\gamma_j$ of positive dimension:
    either the base point free linear system $|\gamma_j|$, or the one dimensional
    family of strict transforms of tangent lines of $\beta_*D$ in Case~2 of
    Proposition~\ref{prop:line}.
     
    We may choose these members so that consecutive ones meet away from $D$.
    Indeed, fix $C'_j$ general in $\mathcal{F}_j$.  By Lemma~\ref{lem:positivity}(3)
    we have $\gamma_j\cdot\gamma_{j+1}>0$, so every member of $\mathcal{F}_{j+1}$
    meets $C'_j$.  The set $C'_j\cap D$ is finite, and for each of its points $y$
    the members of $\mathcal{F}_{j+1}$ through $y$ form a proper closed subset of
    $\mathcal{F}_{j+1}$: for a base point free linear system of positive dimension
    this locus is a hyperplane, and for the family of tangent lines of $\beta_*D$ it
    is finite, since otherwise all tangent lines of $\beta_*D$ would pass through
    $y$ and $\beta_*D$ would be a line.  Hence a general member $C'_{j+1}$ of
    $\mathcal{F}_{j+1}$ meets $C'_j$ only outside $D$.  Choosing the members
    successively and shrinking further, we may also assume that the chosen points
    $x_j\in C'_j\cap C'_{j+1} \setminus D$, $1\leq j\leq s-1$, are pairwise
    distinct.
     
    Lemma~\ref{lem:gluing} now yields a free Campana curve $f\colon C\to X$ with
    \[
    f_*[C]\;=\;\sum_{j=1}^{s}N_j\gamma_j ,
    \]
    which lies in the interior of $\Nef_1(X)$ by Lemma~\ref{lem:positivity}(2).
\end{proof}
     
\begin{corollary}\label{cor:strongly-uniruled}
    Assume the ground field has characteristic $0$.  Then a del Pezzo orbifold with
    irreducible boundary $(X,\epsilon D)$ is strongly Campana uniruled.
\end{corollary}
     
\begin{proof}
    Immediate from Theorem~\ref{thm:free-campana-curve} and Definition~\ref{def:strongly-uniruled}.
\end{proof}
     
\begin{theorem}\label{thm:crc}
    Assume that $\mathbf{k}$ is an algebraically closed field of characteristic $0$
    and let $(X,\epsilon D)$ be a del Pezzo orbifold with irreducible boundary.
    Then $(X,\epsilon D)$ is Campana rationally connected.
\end{theorem}
     
\begin{proof}
    The surface $\X$ is rationally connected and $(X,\epsilon D)$ is strongly
    Campana uniruled by Corollary~\ref{cor:strongly-uniruled}, so
    \cite[Corollary~6.7]{chen2024campana} applies.
\end{proof}
     
\begin{remark}[Reducible boundaries]\label{rem:reducible}
    Lemma~\ref{lem:degree}, Proposition~\ref{prop:conic} and
    Lemma~\ref{lem:positivity} are stated for an arbitrary SNC boundary
    $\Delta_\epsilon=\sum_i\epsilon_i\Delta_i$.  Consequently
    Theorem~\ref{thm:free-campana-curve} holds, with the same proof, for
    \emph{every} klt Campana orbifold $(X,\Delta_\epsilon)$ whose underlying surface
    is a del Pezzo surface of degree at most $6$, or is $\PP^1\times\PP^1$.  Only
    the auxiliary class $H$ of Proposition~\ref{prop:line} uses that the boundary is
    irreducible, so only the three surfaces $\PP^2$, the blow up of $\PP^2$ at one
    point and the del Pezzo surface of degree $7$ require an additional argument in
    the reducible case.  That argument cannot be avoided: for $\X=\PP^2$,
    $\Delta=L_1+\dots+L_4$ four general lines and $\epsilon_i=\frac12$, the pair is
    log Fano, but a general line meets $\Delta$ in $4$ points and the associated
    orbifold $\PP^1$ is Euclidean.  A conic $C'$ tangent to all four lines satisfies Campana conditions: such conics form a pencil, 
    $\mu_k=2$ for $k=1,\dots,4$, so
    \eqref{eq:contact} holds with $N=1$ and no cover is needed, while
    $(-K_{\X}-\Delta)\cdot C'+n-2=-2+4-2=0$. Thus $C'$ is itself a free Campana
    curve, of class $2H$, lying in the interior of $\Nef_1(X)$.
\end{remark}
     
\begin{remark}[On the classification]\label{rmk:classification-not-used}
    The proof above uses neither \cite{dandapat2026classification} nor the
    classification of irreducible boundaries carried out in the appendices: the
    inequality $\Delta_\epsilon\cdot F<-K_{\X}\cdot F$ supplied by ampleness
    replaces the case by case verification.  The classification results are of
    independent interest and are retained in the appendices.
\end{remark}

\section{Weak Approximation}\label{sec:wa}

In this section $k$ is an algebraically closed field of characteristic $0$ and $B$ is a smooth projective curve over $k$ and $K = k(B)$, and we use the 
same notation as in \S~\ref{ss:introweak}.

We recall the following theorem due to Colliot-Th\'el\`elene and Gille which proves weak approximation for for del Pezzo surfaces of degree $d \geq 4$. 
For smooth cubic hypersurfaces of dimension at least two defined over function field of a complex curve \cite[Theorem~1.2]{tian2015weak} proves weak 
approximation.

\begin{theorem}\cite[Section~2]{colliot2004remarques}
    \label{thm:CTG}
    Let $k$ be an algebraically closed field of characteristic $0$, let $B$ be a smooth projective curve over $k$ and let $K = k(B)$. Let $X$ be a 
    smooth del Pezzo surface of degree $d \geq 4$ over $K$. Then $X$ satisfies weak approximation, that is weak approximation holds for $X$ at every 
    place of $K$.
\end{theorem}

Weak approximation is a $K$-birational invariant of smooth geometrically integral $K$-varieties \cite[Proposition~1.1]{colliot2004remarques}, and 
$\PP^2$ satisfies weak approximation. The field $K$ is $C_1$ by Tsen's theorem, so $X(K) \neq \emptyset$. Alternatively, this follows from 
\cite{graber2003families} since a del Pezzo surface is rationally connected. If $d \geq 5$, a del Pezzo surface of degree $d$ with a rational point is
$K$-birational to $\PP^2_K$ (see \cite[Chapter~IV]{manin1986cubic}), and the assertion follows from birational invariance. 

If $d = 4$, the anticanonical embedding realises $X$ as a smooth complete intersection of two quadrics in $\PP^4_K$. Since $K$ is infinite and $X(K)$ is Zariski dense \cite{graber2003families}, we may choose
a $K$-point $P \in X(K)$ lying on none of the sixteen lines of $X_{\overline{K}}$, and blowing up $X$ at $P$ produces a smooth cubic surface $Y \subset \PP^3_K$
containing a $K$-rational line, that is the exceptional curve over $P$. The pencil of planes through that line endows $Y$ with the structure of a conic
bundle over $\PP^1_K$, for which weak approximation holds by \cite[Theorem~2.4]{colliot2004remarques} together with Tsen's theorem. Birational invariance gives
the required result for $X$. If $d = 3$, then it is a smooth cubic surface in $\PP^3$ which is proved in \cite[Theorem~1.2]{tian2015weak}.

For $d \leq 2$ the corresponding statement is open in general \cite[\S~2]{colliot2004remarques}.

\subsection{Campana del Pezzo fibrations}\label{ss:dp-fib}

All Campana fibrations $\pi: (\cX, \Delta_{\epsilon}) \to B$ are assumed to be proper as in \cite[Section~3]{chen2024campana}.

\begin{definition}
    \label{def:dP-fibration}
    Let $\pi: (\cX, \Delta_{\epsilon}) \to B$ be a Campana fibration. We call $\pi$ a \emph{Campana del Pezzo fibration of degree $d$ with irreducible boundary}
    if for general $b \in B$ the fibre $(\cX_b, \Delta_{\epsilon}|_{\cX_b})$, with the weights induced from $\Delta_{\epsilon}$, is a del Pezzo orbifold
    of degree $d$ with irreducible boundary, in the sense of Definition~\ref{def:dP-orbifold}.  
\end{definition}

\begin{lemma}\label{lem:genericfibre}
    Let $\pi:  (\cX,\Delta_\epsilon) \to B$ be a Campana del Pezzo fibration of degree $d$. Then the generic fibre $\cX_K$ is a smooth del Pezzo 
    surface of degree $d$ over $K$.
\end{lemma}

\begin{proof}
    Since $\operatorname{char}k=0$ and the generic fibre of $\pi$ is a smooth projective surface, the generic fibre $\cX_K$ is a smooth projective 
    geometrically integral surface over $K$. Ampleness is an open condition in a proper flat family, so $-K_{\cX_b}$ ample for general $b\in B$. 
    Hence $-K_{\cX_K}$ is ample. Again, $K_{\cX_K}^2 = K_{\cX_b}^2 = d$, since the self-intersection of the relative canonical class is locally 
    constant on the fibres of a smooth proper family of surfaces. Hence $\cX_K$ is a del Pezzo surface of degree $d$ over $K$.
\end{proof}

\subsection{Proof of Theorem~\ref{thmE}}\label{ss:thmE}

\begin{theorem}
    \label{thm:mainWA}
    Let $k$ be an algebraically closed field of characteristic $0$, let $B$ be a smooth projective curve over $k$, and let $\pi: (\cX, \Delta_{\epsilon}) \to B$
    be a Campana del Pezzo fibration of degree $d \geq 3$ with irreducible boundary. Fix closed points $p_1, \ldots, p_r \in \cX$ lying on distinct 
    fibers of $\pi$ and for each $j$, an admissible Campana $n_j$-th jet $\sigma_j$ at $p_j$. Then there is a Campana section of $\pi$ approximating 
    the jet data $\{ p_j, \sigma_j \}$.
\end{theorem}

\begin{proof}
    Let $b \in B$ be general point. By Definition~\ref{def:dP-fibration} the fibre $(\cX_b,\Delta_\epsilon|_{\cX_b})$ is a del Pezzo orbifold with 
    irreducible boundary, so it is strongly Campana uniruled by Theorem~\ref{thmA}, and its underlying surface $\cX_b$ is a del Pezzo surface, hence 
    rationally connected. Thus the hypothesis of Theorem~\ref{thm:WA} on a general fibre of $\pi$ is satisfied.
    
    By Lemma~\ref{lem:genericfibre} the generic fibre $\cX_K$ is a smooth del Pezzo surface of degree $d \geq 4$ over $K = k(B)$, so $\cX_K$ satisfies 
    weak approximation at every place of $K$ by Theorem~\ref{thm:CTG}. Hence Theorem~\ref{thm:WA} applies with $S = \emptyset$ and produces a Campana
    section approximating $\{ p_j, \sigma_j \}$.
\end{proof}

\begin{corollary}\label{cor:existence}
    In the situation of Theorem~\ref{thm:mainWA}, Campana sections of $\pi$ exist. Moreover, for every point $p \in \cX \setminus \Delta$ at which 
    $\pi$ is smooth there exists a Campana section of $\pi$ passing through $p$.  In particular the union of the images of the Campana sections of 
    $\pi$ is Zariski dense in $\cX$.
\end{corollary}
    
\begin{proof}
    Take $r = 0$ in Theorem~\ref{thm:mainWA}, that is trivial jet data, produces a Campana section of $\pi$.
    
    Now, let $p \in \cX \setminus \Delta$ be a point at which $\pi$ is smooth, and let $\sigma$ be the $0$-jet at $p$. Since $\pi$ is smooth at $p$, 
    the jet $\sigma$ is admissible. Since $p \notin \Delta$, the local intersection multiplicity of $\sigma$ with every boundary component of $\Delta$ 
    vanishes, so $\sigma$ trivially satisfies the Campana condition of \cite[Definition~6.2]{chen2024campana} and is an admissible Campana jet. 
    By Theorem~\ref{thm:mainWA} there exists a Campana section of $\pi$ inducing $\sigma$, that is passing through $p$.
    
    Now, since $\pi$ is smooth over a dense open subset of $B$ and $\Delta$ is a proper closed subset of $\cX$, the set of points $p$ as above is 
    a dense open subset of $\cX$. 
\end{proof}

\begin{remark}(Lower degree del Pezzo orbifolds)
    \label{rem:arbdegree}
    The degree $d \geq 3$ is used only due to application of Theorem~\ref{thm:CTG}. The proof of Theorem~\ref{thm:mainWA} therefore gives the following 
    statement in arbitrary degree.

    \emph{Let $\pi: (\cX,\Delta_\epsilon) \to B$ be a Campana del Pezzo fibration with irreducible boundary, and let $S$ be a finite set of places of 
    $K=k(B)$ such that the generic fibre $\cX_K$ satisfies weak approximation at every finite set of places disjoint from $S$. Then every finite 
    collection of admissible Campana jets supported on distinct fibres over $B \setminus S$ is induced by a Campana section of $\pi$.}

    One may take $S = \emptyset$ in each of the following cases:
    \begin{enumerate}
        \item $d \geq 4$, by \cite{colliot2004remarques}, this is Theorem~\ref{thm:mainWA}.
        \item $d = 3$ and $\cX_K$ is a smooth cubic surface with square-free discriminant, by \cite{hassett2008approximation}. More generally, for any 
        smooth cubic hypersurface by \cite{tian2015weak}.
        \item $d = 2$ and $\cX_K$ has square-free discriminant, by \cite{knecht2013weak}.
        \item $d = 1, 2$ and $\pi$ can be completed to a sufficiently generic family in the relevant parameter space over $B$, by \cite{xu12}.
    \end{enumerate}
    In general we may take $S$ to be the set of places of bad reduction, by \cite[Theorem~3]{hassett2005}, that is Corollary~\ref{corC}. The conjecture 
    of Hassett and Tschinkel \cite[Section~1]{hassett2005} that every smooth rationally connected variety over $K$ satisfies weak approximation at all 
    places would imply, in combination with Theorem~\ref{thmA} and Theorem~\ref{thm:WA}, weak approximation at all places for Campana sections of 
    \emph{every} Campana del Pezzo fibration with irreducible boundary.
\end{remark}

\appendix

\section{Surfaces of degree 6}\label{app:six}

The classification carried out in the appendices are based on the following elementary observations. Throughout, $X$ is a del Pezzo surface of degree 
$d = K_X^2$ and $\epsilon = 1 - 1/m$ with $m \in \mathbb{Z}_{\geq 2}$, so that $\tfrac12 \le \epsilon<1$. We write $\mathcal{R}(X)$ for the finite set 
of generators of the extremal rays of $\mathrm{Eff}(X)$ listed in Lemma~\ref{lem:dP-cones}. Recall that on a surface $\mathrm{Eff}(X)=\overline{NE}(X)$, 
so a divisor is nef if and only if it is non-negative on $\mathcal{R}(X)$, and ample if and only if it is positive on $\mathcal{R}(X)$, the ample cone 
being the interior of the nef cone. Also recall that every irreducible curve $D$ on a del Pezzo surface $X$ is either a $(-1)$-curve or nef. 

\begin{lemma}\label{lem:criterion}
    Let $D$ be a non-zero effective divisor on $X$ and let $0 < \epsilon <1$.
    \begin{enumerate}
        \item $-(K_X+\epsilon D)$ is nef (resp.\ ample) if and only if $\epsilon\,(D\cdot R)\leq -K_X\cdot R$ (resp.\ $<$) for every 
        $R\in\mathcal{R}(X)$.
        \item Put $$\tau(D):=\min\Bigl\{\tfrac{-K_X\cdot R}{D\cdot R}\ :\ R \in \mathcal{R}(X),\ D\cdot R>0\Bigr\}\in\mathbb{Q}_{>0}\cup\{+\infty\},$$
        the minimum over the empty set being $+\infty$. Then $-(K_X+\epsilon D)$ is ample if and only if $\epsilon<\tau(D)$, and nef if and only 
        if $\epsilon\le\tau(D)$.
        \item If $-(K_X+\epsilon D)$ is nef, then it is big if and only if $d-2\epsilon\,(-K_X\cdot D)+\epsilon^{2}D^{2}>0$, and
        $$-K_X\cdot D\ \le\ \frac{d}{\epsilon}\ \le\ 2d \qquad\text{whenever }\epsilon\ge\tfrac12 .$$
    \end{enumerate}
\end{lemma}

\begin{proof}
    The proof is elementary.
\end{proof}

\begin{corollary}\label{cor:mu}
    Assume $1\le d\le7$, so that $\mathcal{R}(X)$ is the set of $(-1)$-curves, and let $D$ be an irreducible curve on $X$. Put
    $\mu(D):=\max\{D\cdot E\ :\ E\ \text{a }(-1)\text{-curve on }X\}$.  Then $\mu(D)\ge1$, $\tau(D)=1/\mu(D)$, and for $\epsilon=1-1/m$ with $m\ge2$:
    \begin{enumerate}
        \item $-(K_X+\epsilon D)$ is ample for every $m\ge2$ if and only if $\mu(D)=1$;
        \item $-(K_X+\epsilon D)$ is nef but not ample if and only if $\mu(D)=2$ and
            $m=2$;
        \item $-(K_X+\epsilon D)$ is not nef for any $m\ge2$ if and only if
            $\mu(D)\ge3$.
    \end{enumerate}
\end{corollary}

\begin{proof}
    Since $-K_X\cdot E = 1$ for every $(-1)$-curve $E$, By Lemma~\ref{lem:criterion}(2), $\tau(D)=1/\mu(D)$ when $\mu(D)\ge1$. If $D\cdot E\le0$ 
    for every $(-1)$-curve $E$ then, by writing the effective class $-K_X$ as a non-negative combination of $(-1)$-curves, we get $-K_X\cdot D\le0$, 
    contradicting the ampleness of $-K_X$, hence $\mu(D) \ge 1$. Now (1)-(3) follow from $\tfrac12\le\epsilon<1$: we have $\epsilon<1=\tau(D)$ always 
    when $\mu(D)=1$, $\epsilon\le\tfrac12=\tau(D)$ exactly for $m=2$ when $\mu(D)=2$, and $\tau(D)\le\tfrac13<\epsilon$ when $\mu(D)\ge3$.
\end{proof}

We also record the following elementary fact about a del Pezzo surface of degree $d \geq 3$.

\begin{lemma}\label{lem:irr-members}
    Let $X$ be a del Pezzo surface of degree $d \ge 3$ and let $D$ be a non-zero nef class on $X$.
    \begin{enumerate}
        \item $|D|$ is base point free.
        \item If $D^{2}>0$, the general member of $|D|$ is a smooth irreducible curve.
        \item If $D^{2}=0$, then $D=kF$ with $F$ a smooth fiber of a conic fibration and $k\ge1$, the class $D$ is represented by an irreducible curve 
        if and only if $k=1$.
    \end{enumerate}
\end{lemma}

Let $X$ be a del Pezzo surface of degree $6$, that is the blow up $\beta: X \to \mathbb{P}^{2}$ at three points in general position, with
$K_X=-3H+E_1+E_2+E_3$. There are six $(-1)$-curves on $X$, namely $E_1,E_2,E_3$ and $L_{ij} := H-E_i-E_j$ for $1\le i<j\le3$. Throughout, $i,j,k$ denote 
pairwise distinct elements of $\{1,2,3\}$, $D$ is an irreducible curve on $X$, and $\epsilon=1-1/m$ with $m\in\mathbb{Z}_{\ge2}$.
In Table~\ref{tab:d6} we use the abbreviations
\begin{itemize}
    \item[\textup{(a)}] $-(K_X+\epsilon D)$ is ample for every $m\ge2$.
    \item[\textup{(c)}] $-(K_X+\epsilon D)$ is nef if and only if $m=2$, never ample, and big for $m=2$.
    \item[\textup{(d)}] $-(K_X+\epsilon D)$ is nef if and only if $m=2$, never ample, and not big.
\end{itemize}

\begin{proposition}\label{prop:dp6}
    Let $X$ be a del Pezzo surface of degree $6$, let $D \subset X$ be an irreducible curve and let $\epsilon=1-1/m$ with $m\ge2$. If 
    $-(K_X+\epsilon D)$ is nef, then the class of $D$ occurs in Table~\textup{\ref{tab:d6}}, and the positivity of $-(K_X+\epsilon D)$ is the
    one recorded there.  Conversely every class in the table is represented by an irreducible curve. In particular $(X,\epsilon D)$ is a del Pezzo 
    orbifold exactly for the seven classes with $\mu(D)=1$, for every $m\ge2$, and $-(K_X+\epsilon D)$ is nef but not ample exactly for the 
    thirteen classes with $\mu(D)=2$, and then only for $m=2$.
\end{proposition}

\begin{table}[ht]
    \small
    \begin{tabular}{@{}lcccc@{}}
        \hline
        $D$ & $-K_X\cdot D$ & $D^{2}$ & $\mu(D)$ & $-(K_X+\epsilon D)$\\
        \hline
        $E_i$                      & $1$  & $-1$ & $1$ & (a)\\
        $H-E_i-E_j$                & $1$  & $-1$ & $1$ & (a)\\
        $H-E_i$                    & $2$  & $0$  & $1$ & (a)\\
        $H$                        & $3$  & $1$  & $1$ & (a)\\
        $2H-E_1-E_2-E_3$           & $3$  & $1$  & $1$ & (a)\\
        $2H-E_i-E_j$               & $4$  & $2$  & $1$ & (a)\\
        $-K_X=3H-E_1-E_2-E_3$      & $6$  & $6$  & $1$ & (a)\\
        \hline
        $2H-E_i$                   & $5$  & $3$  & $2$ & (c)\\
        $3H-2E_i-E_j-E_k$          & $5$  & $3$  & $2$ & (c)\\
        $2H$                       & $6$  & $4$  & $2$ & (c)\\
        $3H-2E_i-E_j$              & $6$  & $4$  & $2$ & (c)\\
        $4H-2E_1-2E_2-2E_3$        & $6$  & $4$  & $2$ & (c)\\
        $3H-E_i-E_j$               & $7$  & $7$  & $2$ & (c)\\
        $4H-2E_i-2E_j-E_k$         & $7$  & $7$  & $2$ & (c)\\
        $4H-2E_i-E_j-E_k$          & $8$  & $10$ & $2$ & (c)\\
        $4H-2E_i-2E_j$             & $8$  & $8$  & $2$ & (d)\\
        $4H-E_1-E_2-E_3$           & $9$  & $13$ & $2$ & (c)\\
        $5H-2E_1-2E_2-2E_3$        & $9$  & $13$ & $2$ & (c)\\
        $5H-2E_i-2E_j-E_k$         & $10$ & $16$ & $2$ & (d)\\
        $-2K_X=6H-2E_1-2E_2-2E_3$  & $12$ & $24$ & $2$ & (d)\\
        \hline
    \end{tabular}
    \medskip
    \caption{Irreducible boundaries on a del Pezzo surface of degree $6$.}
    \label{tab:d6}
\end{table}

\begin{proof}
    Write $D=\alpha H-\beta_1E_1-\beta_2E_2-\beta_3E_3$, so that $D\cdot E_i=\beta_i$ and $D\cdot L_{ij}=\alpha-\beta_i-\beta_j$. By
    Corollary~\ref{cor:mu} we may assume that $\mu(D)\le2$, and it is enough to determine all such classes.
     
    If $D$ is a $(-1)$-curve then $D\in\{E_i,L_{ij}\}$ and we have $E_i\cdot L_{ij}=1$ while all other products of distinct $(-1)$-curves vanish.
    Hence $\mu(D)=1$, which gives the first two rows.
     
    Otherwise $D$ is nef, so $\beta_i\ge0$ and $\alpha\ge\beta_i+\beta_j$ for all $i\neq j$. Reordering the
    $E_i$ we may assume $\beta_1\ge\beta_2\ge\beta_3\ge0$, then
    $\mu(D)=\max\{\beta_1,\ \alpha-\beta_2-\beta_3\}$, and $\mu(D)\le2$ is
    equivalent to $$\beta_1\le2,\qquad \beta_1+\beta_2\ \le\ \alpha\ \le\ \beta_2+\beta_3+2 .$$
    There are ten admissible triples $(\beta_1,\beta_2,\beta_3)$ and at most three
    values of $\alpha$ for each, giving twenty classes: those of the table together
    with $D=0$ and $D=2(H-E_1)$.  The latter contains no irreducible curve by
    Lemma~\ref{lem:irr-members}(3).  Every remaining class is either a conic class
    or a nef class with $D^{2}>0$, hence is represented by an irreducible curve by
    Lemma~\ref{lem:irr-members}.  The columns $-K_X\cdot D$, $D^{2}$ and $\mu(D)$
    are immediate, and the last column follows from Corollary~\ref{cor:mu} together
    with the bigness criterion of Lemma~\ref{lem:criterion}(3), which for $m=2$
    gives $6-(-K_X\cdot D)+\tfrac14D^{2}>0$.
\end{proof}

\begin{remark}\label{rem:d6-geom}
    The seven classes with $\mu(D)=1$ are, geometrically the $(-1)$-curves, the
    fibre classes $H-E_i$ of the three conic bundle structures, the pullbacks of a
    line under the two birational morphisms $X\to\mathbb{P}^{2}$, which are $H$ and
    $2H-E_1-E_2-E_3$, the pullbacks $2H-E_i-E_j$ of a $(1,1)$-class under a
    birational morphism $X\to\mathbb{P}^{1}\times\mathbb{P}^{1}$, and the
    anticanonical class. Both blocks of Table~\ref{tab:d6} are invariant under the
    Weyl group $W(A_2\times A_1)$ of $X$, which acts by permuting $E_1,E_2,E_3$ and
    by the Cremona involution $H\mapsto2H-E_1-E_2-E_3$, $E_i\mapsto H-E_j-E_k$. The
    classes in the second block form nine $W$-orbits.
\end{remark}

\section{Surfaces of degree 7}\label{app:seven}

Let $X$ be a del Pezzo surface of degree $7$, that is the blow up $\beta: X\to\mathbb{P}^{2}$ at two distinct points, with
$K_X=-3H+E_1+E_2$. There are three $(-1)$-curves on $X$, namely $E_1$, $E_2$
and $L:=H-E_1-E_2$. Throughout, $\{i,j\}=\{1,2\}$, $D$ is an irreducible curve
on $X$ and $\epsilon=1-1/m$ with $m\in\mathbb{Z}_{\ge2}$. The abbreviations
\textup{(a)}, \textup{(c)}, \textup{(d)} are those of Appendix~\ref{app:six}.

\begin{proposition}\label{prop:dp7}
    Let $X$ be a del Pezzo surface of degree $7$, let $D\subset X$ be an
    irreducible curve and let $\epsilon=1-1/m$ with $m\ge2$.  If
    $-(K_X+\epsilon D)$ is nef, then the class of $D$ occurs in
    Table~\textup{\ref{tab:d7}}, and the positivity of $-(K_X+\epsilon D)$ is the
    one recorded there.  Conversely every class in the table is represented by an
    irreducible curve.  In particular $(X,\epsilon D)$ is a del Pezzo orbifold
    exactly for the seven classes with $\mu(D)=1$, for every $m\ge2$, and
    $-(K_X+\epsilon D)$ is nef but not ample exactly for the eleven classes with
    $\mu(D)=2$, and then only for $m=2$.
\end{proposition}

\begin{table}[ht]
    \small
    \begin{tabular}{@{}lcccc@{}}
        \hline
        $D$ & $-K_X\cdot D$ & $D^{2}$ & $\mu(D)$ & $-(K_X+\epsilon D)$\\
        \hline
        $E_i$                  & $1$  & $-1$ & $1$ & (a)\\
        $H-E_1-E_2$            & $1$  & $-1$ & $1$ & (a)\\
        $H-E_i$                & $2$  & $0$  & $1$ & (a)\\
        $H$                    & $3$  & $1$  & $1$ & (a)\\
        $2H-E_1-E_2$           & $4$  & $2$  & $1$ & (a)\\
        $2H-E_i$               & $5$  & $3$  & $1$ & (a)\\
        $-K_X=3H-E_1-E_2$      & $7$  & $7$  & $1$ & (a)\\
        \hline
        $2H$                   & $6$  & $4$  & $2$ & (c)\\
        $3H-2E_i-E_j$          & $6$  & $4$  & $2$ & (c)\\
        $3H-2E_i$              & $7$  & $5$  & $2$ & (c)\\
        $3H-E_i$               & $8$  & $8$  & $2$ & (c)\\
        $4H-2E_1-2E_2$         & $8$  & $8$  & $2$ & (c)\\
        $4H-2E_i-E_j$          & $9$  & $11$ & $2$ & (c)\\
        $4H-2E_i$              & $10$ & $12$ & $2$ & (d)\\
        $4H-E_1-E_2$           & $10$ & $14$ & $2$ & (c)\\
        $5H-2E_1-2E_2$         & $11$ & $17$ & $2$ & (c)\\
        $5H-2E_i-E_j$          & $12$ & $20$ & $2$ & (d)\\
        $-2K_X=6H-2E_1-2E_2$   & $14$ & $28$ & $2$ & (d)\\
        \hline
    \end{tabular}
    \medskip
    \caption{Irreducible boundaries on a del Pezzo surface of degree $7$.}
    \label{tab:d7}
\end{table}

\begin{proof}
    Write $D=\alpha H-\beta_1E_1-\beta_2E_2$, so that $D\cdot E_i=\beta_i$ and
    $D\cdot L=\alpha-\beta_1-\beta_2$. Hence
    $\mu(D)=\max\{\beta_1,\beta_2,\alpha-\beta_1-\beta_2\}$ when $D$ is nef. By
    Corollary~\ref{cor:mu} we may assume $\mu(D)\le2$.
     
    If $D$ is a $(-1)$-curve then $D\in\{E_1,E_2,L\}$, since $E_1\cdot E_2=0$ and
    $E_i\cdot L=1$ we get $\mu(D)=1$, which gives the first two rows.  Otherwise
    $D$ is nef, so $\beta_1,\beta_2\ge0$ and
    $\alpha\ge\beta_1+\beta_2$, assuming $\beta_1\ge\beta_2$ after relabelling, the
    condition $\mu(D)\le2$ becomes
    $$\beta_1\le2,\qquad \beta_1+\beta_2\ \le\ \alpha\ \le\ \beta_1+\beta_2+2 .$$
    This leaves six pairs $(\beta_1,\beta_2)$ and three values of $\alpha$ for each,
    that is, eighteen classes: those of the table together with $D=0$ and
    $D=2(H-E_1)$, the latter carrying no irreducible curve by
    Lemma~\ref{lem:irr-members}(3).  All remaining classes are conic classes or nef
    classes with $D^{2}>0$, hence are represented by irreducible curves by
    Lemma~\ref{lem:irr-members}.  The last column follows from
    Corollary~\ref{cor:mu} and from Lemma~\ref{lem:criterion}(3), which for $m=2$
    gives $7-(-K_X\cdot D)+\tfrac14D^{2}>0$.
\end{proof}

\begin{remark}\label{rem:d7-geom}
    The classes with $\mu(D)=1$ are the three $(-1)$-curves, the two conic classes
    $H-E_i$, the pullback $H$ of a line, the pullback $2H-E_1-E_2$ of a
    $(1,1)$-class on $\mathbb{P}^{1}\times\mathbb{P}^{1}$, the pullbacks
    $2H-E_i$ of the class $2H'-E$ on the blow up of $\mathbb{P}^{2}$ at one point,
    and the anticanonical class.
\end{remark}

\section{Surfaces of degree 8}\label{app:eight}

A del Pezzo surface of degree $8$ is either the blow up of $\mathbb{P}^{2}$ at
one point or $\mathbb{P}^{1}\times\mathbb{P}^{1}$. The two cases are treated in
Propositions~\ref{prop:dp8-1} and~\ref{prop:dp8-2}. In both cases, $D$ is an
irreducible curve and $\epsilon=1-1/m$ with $m\in\mathbb{Z}_{\ge2}$.  Since
$\mathrm{Eff}(X)$ now has a generator of anticanonical degree $2$, the
trichotomy of Corollary~\ref{cor:mu} acquires a fourth case, and we use the
abbreviations \textup{(a)}, \textup{(c)}, \textup{(d)} of
Appendix~\ref{app:six} together with
\begin{itemize}
    \item[\textup{(b)}] $-(K_X+\epsilon D)$ is ample if and only if $m=2$, nef if and only if $m\in\{2,3\}$, and big if and only if $m=2$.
\end{itemize}

\subsection{The blow up of \texorpdfstring{$\mathbb{P}^{2}$}{P2} at one point}
 
Here $K_X=-3H+E$, where $H$ is the pullback of the hyperplane class and $E$ the
exceptional curve, and $\mathrm{Eff}(X)$ is generated by $E$ and by the fibre
class $f:=H-E$, with $-K_X\cdot E=1$ and $-K_X\cdot f=2$.
 
\begin{proposition}\label{prop:dp8-1}
    Let $X$ be the blow up of $\mathbb{P}^{2}$ at one point, let $D\subset X$ be an
    irreducible curve and let $\epsilon=1-1/m$ with $m\ge2$.  If
    $-(K_X+\epsilon D)$ is nef, then the class of $D$ occurs in
    Table~\textup{\ref{tab:d8blowup}}, and the positivity of $-(K_X+\epsilon D)$ is
    the one recorded there.  Conversely every class in the table is represented by
    an irreducible curve.  In particular $(X,\epsilon D)$ is a del Pezzo orbifold
    exactly for the six classes marked \textup{(a)}, for every $m\ge2$, and for the
    two classes marked \textup{(b)}, for $m=2$.
\end{proposition}

\begin{table}[ht]
    \small
    \begin{tabular}{@{}lccccc@{}}
        \hline
        $D$ & $-K_X\cdot D$ & $D^{2}$ & $D\cdot E$ & $D\cdot f$ & $-(K_X+\epsilon D)$\\
        \hline
        $E$              & $1$  & $-1$ & $-1$ & $1$ & (a)\\
        $f=H-E$          & $2$  & $0$  & $1$  & $0$ & (a)\\
        $H$              & $3$  & $1$  & $0$  & $1$ & (a)\\
        $2H-E$           & $5$  & $3$  & $1$  & $1$ & (a)\\
        $2H$             & $6$  & $4$  & $0$  & $2$ & (a)\\
        $-K_X=3H-E$      & $8$  & $8$  & $1$  & $2$ & (a)\\
        \hline
        $3H$             & $9$  & $9$  & $0$  & $3$ & (b)\\
        $4H-E$           & $11$ & $15$ & $1$  & $3$ & (b)\\
        \hline
        $3H-2E$          & $7$  & $5$  & $2$  & $1$ & (c)\\
        $4H-2E$          & $10$ & $12$ & $2$  & $2$ & (c)\\
        $5H-2E$          & $13$ & $21$ & $2$  & $3$ & (c)\\
        $4H$             & $12$ & $16$ & $0$  & $4$ & (d)\\
        $5H-E$           & $14$ & $24$ & $1$  & $4$ & (d)\\
        $-2K_X=6H-2E$    & $16$ & $32$ & $2$  & $4$ & (d)\\
        \hline
    \end{tabular}
    \medskip
    \caption{Irreducible boundaries on the blow up of $\mathbb{P}^{2}$ at one
    point.}
    \label{tab:d8blowup}
\end{table}

\begin{proof}
    Write $D=\alpha H-\beta E$, so $D\cdot E=\beta$ and $D\cdot f=\alpha-\beta$, by
    Lemma~\ref{lem:criterion}(2),
    $\tau(D)=\min\{1/\beta,\ 2/(\alpha-\beta)\}$, with the convention that a factor
    is omitted when the corresponding intersection number is $\le0$.  Since
    $\epsilon\ge\tfrac12$, nefness forces $\beta\le2$ and $\alpha-\beta\le4$.
     
    If $D$ is a $(-1)$-curve then $D=E$, the first row.  Otherwise $D$ is nef. So $0\le\beta\le2$ and
    $0\le\alpha-\beta\le4$, which leaves 13 nef classes of the table together
    with $D=0$ and $D=2f$. The class $2f$ carries no irreducible curve by
    Lemma~\ref{lem:irr-members}(3), and all others are the fibre class $f$ or nef
    classes with $D^{2}>0$, hence are represented by irreducible curves.  Finally
    $\tau(D) \geq 1$ in the block (a), $\tau(D)=\tfrac23$ in the block (b) and
    $\tau(D)=\tfrac12$ in the blocks (c), (d). The assertions about ampleness and
    nefness follow from Lemma~\ref{lem:criterion}(2), and those about bigness from
    Lemma~\ref{lem:criterion}(3), which gives
    $8-2\epsilon(-K_X\cdot D)+\epsilon^{2}D^{2}>0$.
\end{proof}

\subsection{The quadric surface}
 
Here $X=\mathbb{P}^{1}\times\mathbb{P}^{1}$, $\mathrm{Pic}(X)=\mathbb{Z}H_1\oplus
\mathbb{Z}H_2$ with $H_1^{2}=H_2^{2}=0$ and $H_1\cdot H_2=1$, and
$-K_X=2H_1+2H_2$.  We say that $D$ has \emph{type} $(a,b)$ if
$D\equiv aH_1+bH_2$. Thus $D\cdot H_1=b$, $D\cdot H_2=a$ and
$-K_X\cdot D=2(a+b)$. Both generators $H_1,H_2$ of $\mathrm{Eff}(X)$ have
anticanonical degree $2$, so $\tau(D)=2/\max\{a,b\}$.

\begin{proposition}\label{prop:dp8-2}
    Let $X=\mathbb{P}^{1}\times\mathbb{P}^{1}$, let $D\subset X$ be an irreducible
    curve of type $(a,b)$ with $a\ge b\ge0$, and let $\epsilon=1-1/m$ with $m\ge2$.
    Then $b\ge1$ unless $(a,b)=(1,0)$, and:
    \begin{enumerate}
        \item if $a\le2$, that is $(a,b)\in\{(1,0),(1,1),(2,1),(2,2)\}$, then $-(K_X+\epsilon D)$ is ample for every $m \ge 2$.
        \item if $a=3$, that is $(a,b)\in\{(3,1),(3,2),(3,3)\}$, then $-(K_X+\epsilon D)$ is ample if and only if $m=2$, nef if and only if
        $m\in\{2,3\}$, and big if and only if $m=2$.
        \item if $a=4$, that is $(a,b)\in\{(4,1),(4,2),(4,3),(4,4)\}$, then $-(K_X+\epsilon D)$ is nef if and only if $m=2$, and it is then neither
        ample nor big.
        \item if $a\ge5$, then $-(K_X+\epsilon D)$ is not nef for any $m\ge2$.
    \end{enumerate}
    Every type listed in \textup{(1)-(3)} is represented by a smooth irreducible curve.
\end{proposition}
     
\begin{proof}
    An effective divisor of type $(a,0)$ with $a\ge2$ is a sum of $a$ fibres of the
    second projection, hence it is not irreducible. This proves the first statement,
    and conversely $|aH_1+bH_2|$ is very ample for $a,b\ge1$, so its general member
    is smooth and irreducible. By Lemma~\ref{lem:criterion}(2) and
    $\tau(D)=2/\max\{a,b\}=2/a$, the divisor $-(K_X+\epsilon D)$ is ample for
    $\epsilon<2/a$ and nef for $\epsilon\le2/a$, with $\epsilon=1-1/m$ this gives
    (1)-(4), since $\epsilon<2/3$ if and only if $m=2$ and $\epsilon\le2/3$ if and
    only if $m\in\{2,3\}$, while $\epsilon\le\tfrac12$ if and only if $m=2$.
    Finally $-(K_X+\epsilon D)\equiv(2-\epsilon a)H_1+(2-\epsilon b)H_2$ has
    self-intersection $2(2-\epsilon a)(2-\epsilon b)$, which is positive in case (1),
    positive in case (2) exactly for $m=2$, and zero in case (3).
\end{proof}

\bibliographystyle{amsalpha}
\bibliography{ref}

\end{document}